\documentclass[a4paper]{article}
\pdfoutput=1
\usepackage{amsmath}
\usepackage{stmaryrd}
\usepackage{graphicx}
\usepackage{amssymb}
\usepackage{algorithm}
\usepackage{booktabs}
\usepackage{multirow}
\usepackage{algpseudocode}

\newtheorem{corollary}{Corollary}
\newtheorem{lemma}{Lemma}
\newtheorem{proof}{Proof}
\usepackage{subfigmat}
\usepackage{bm}
\usepackage{amsmath,amsfonts,mathtools}
\usepackage[width=7in]{geometry}
\providecommand{\mat}[1]{\bm{#1}}%
\renewcommand{\vec}[1]{\mathbf{#1}}

\providecommand{\mA}{\ensuremath{\mat{A}}}

\providecommand{\mD}{\ensuremath{\mat{D}}}

\providecommand{\mP}{\ensuremath{\mat{P}}}
\providecommand{\mQ}{\ensuremath{\mat{Q}}}
\providecommand{\mR}{\ensuremath{\mat{R}}}
\providecommand{\mS}{\ensuremath{\mat{S}}}

\providecommand{\mV}{\ensuremath{\mat{V}}}

\providecommand{\va}{\ensuremath{\vec{a}}}
\providecommand{\vb}{\ensuremath{\vec{b}}}
\providecommand{\vc}{\ensuremath{\vec{c}}}

\providecommand{\vg}{\ensuremath{\vec{g}}}
\providecommand{\vh}{\ensuremath{\vec{h}}}

\providecommand{\vu}{\ensuremath{\vec{u}}}

\providecommand{\vx}{\ensuremath{\vec{x}}}

\providecommand{\vz}{\ensuremath{\vec{z}}}

\title{
    Effectively Subsampled Quadratures for \\ Least Squares Polynomial Approximations
}

\author{Pranay Seshadri \thanks{Research Associate, Department of Engineering, University of Cambridge, U.K., \texttt{ps583@cam.ac.uk}}, Akil Narayan \thanks{Assistant Professor, Department of Mathematics, University of Utah, Salt Lake City, U.S.A.}, Sankaran Mahadevan \thanks{John R. Murray Sr. Professor, Department of Civil and Environmental Engineering, Vanderbilt University, Nashville, U.S.A.} }

\begin{document}

\maketitle

\begin{abstract}
This paper proposes a new deterministic sampling strategy for constructing polynomial chaos approximations for expensive physics simulation models. The proposed approach, \emph{effectively subsampled quadratures} involves sparsely subsampling an existing tensor grid using QR column pivoting. For polynomial interpolation using hyperbolic or total order sets, we then solve the following \emph{square} least squares problem.  For polynomial approximation, we use a column pruning heuristic that removes columns based on the highest total orders and then solves the \emph{tall} least squares problem. While we provide bounds on the condition number of such tall submatrices, it is difficult to ascertain how column pruning effects solution accuracy as this is problem specific. We conclude with numerical experiments on an analytical function and a model piston problem that show the efficacy of our approach compared with randomized subsampling. We also show an example where this method fails.
\end{abstract}

\section{Introduction \& motivation}
\label{sec:intro}
\noindent Polynomial chaos is a powerful tool for uncertainty quantification that has seen widespread use in numerous disciplines \cite{Seshadri_LEAK,  Aeroelas, mechanical, chemical}. It approximates a model's response with respect to uncertainty in the input parameters using orthogonal polynomials. One important challenge in computing these polynomial approximations lies in determining their coefficients. Standard approaches to compute these coefficients include tensor grids as in Figure~\ref{figure_intro}(a) and sparse grids in (b). In this paper we present a deterministic strategy to construct stable polynomial approximations by \emph{subsampling points} from a tensor grid, as in Figure~\ref{figure_intro}(c) . 

The motivation for moving away from tensor product spaces is that the number of points grow exponentially with dimension. While sparse grids \cite{Sparse1, Sparse2, SPAM} are one way forward, they still restrict users to very specific index sets---even with their various growth rules. Least squares based methods on the other hand offer more flexibility. The idea of employing least squares techniques in a polynomial approximation context is not a new idea, but recent theoretical results \cite{Zhou} motivate the idea of randomized or \textit{subsampled} quadratures. In a nutshell the idea is to evaluate the simulation at only a small subset of points from a suitable tensorial grid,  and to subsequently construct an approximation using either compressed sensing \cite{Doos_1, Hampton, Peng, Tang} or least squares. The goal is to subsample in a way such that the subsample-constructed surrogate approximates the surrogate that would obtained by using the all the points from the tensor grid, by minimizing the mean-square error in the coefficients. After subsampling, the approximation scheme we focus on is a least-squares-type approach; subsampling approaches for least-squares have received little attention \cite{Zhou} compared to compressive sampling.

\begin{figure}
\begin{subfigmatrix}{3}
\subfigure[]{\includegraphics{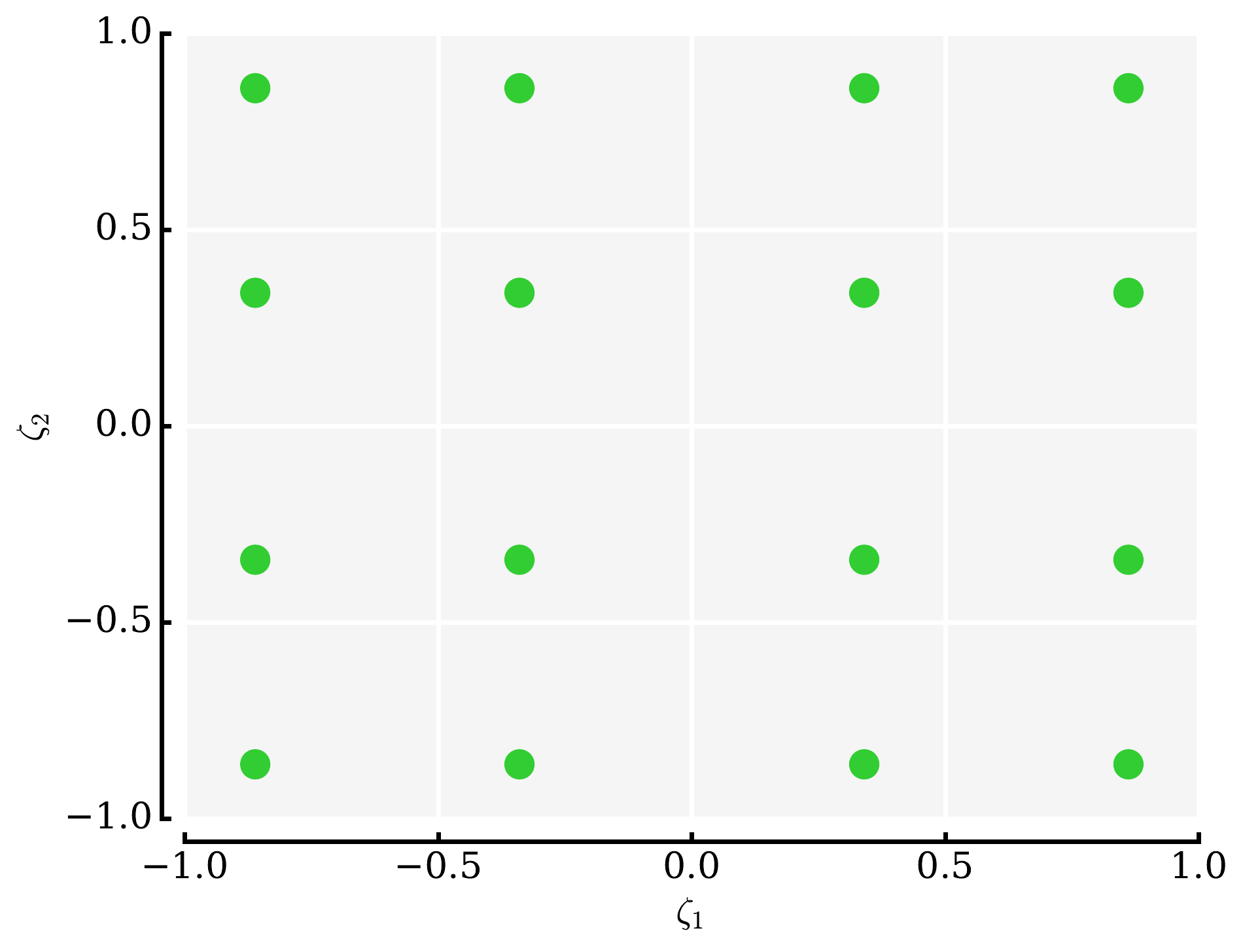}}
\subfigure[]{\includegraphics{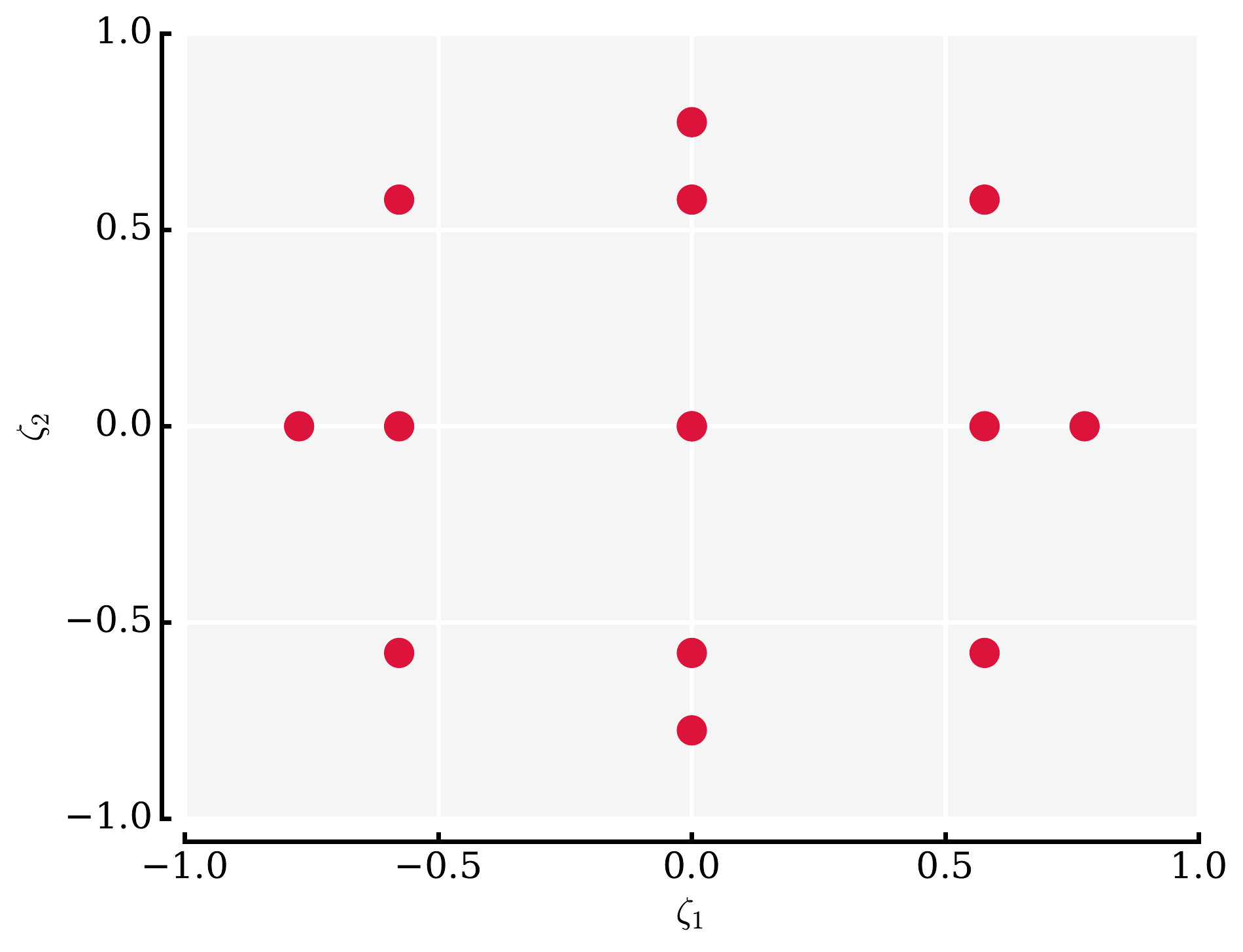}}
\subfigure[]{\includegraphics{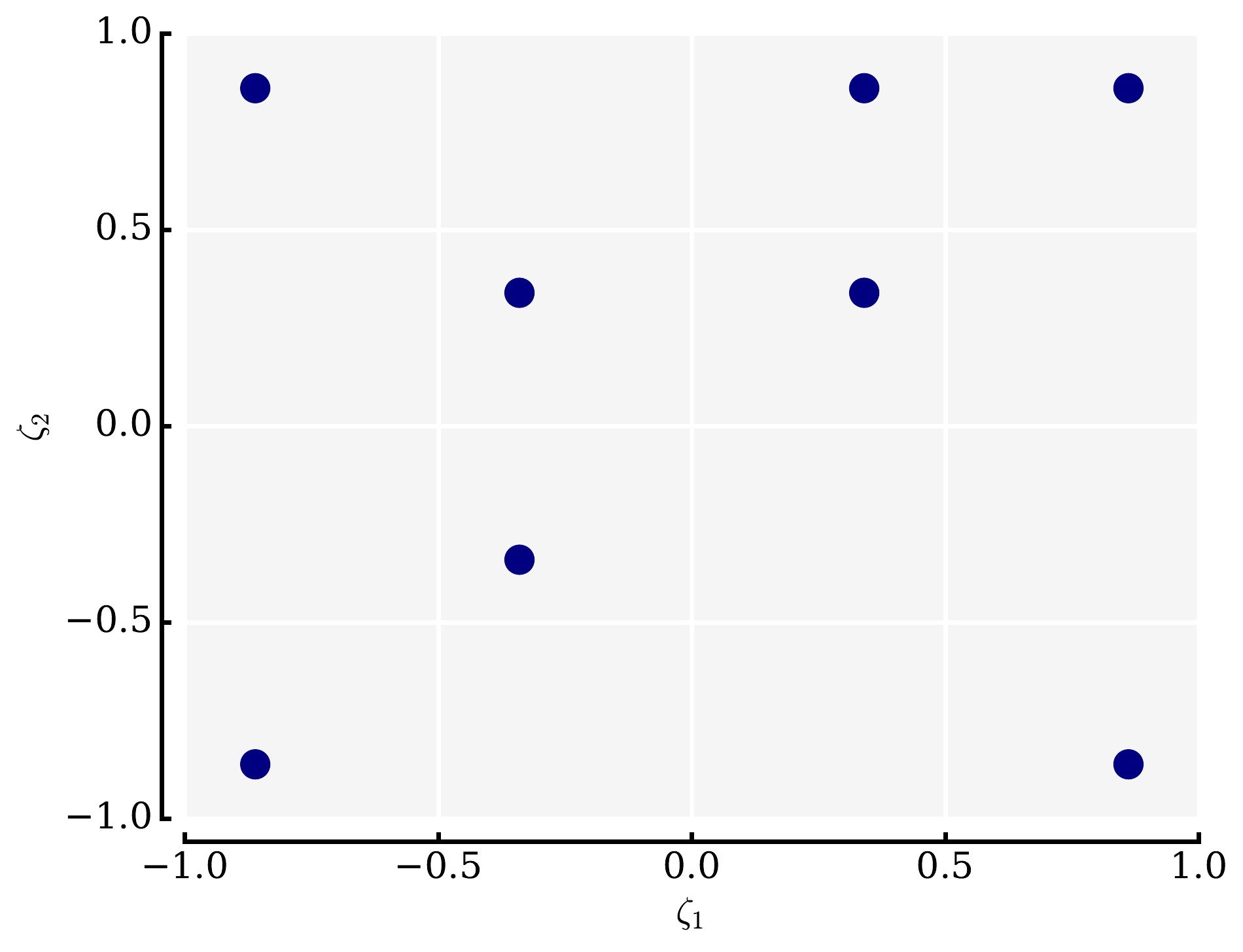}}
\end{subfigmatrix}
\caption{Sample stencils for (a) Tensor grid; (b) Sparse grid and (c) Effectively subsampled quadratures, with a maximum univariate degree of 4 using Gauss-Legendre quadrature points.}
\label{figure_intro}
\end{figure}

Despite ease of implementation, discrete least-squares has known stability issues. For instance, when using abscissae associated with Newton-Cotes quadrature (equidistant point sets), it is relatively unstable even for an infinitely smooth noiseless function \cite{platte_impossibility_2011}. Recent theoretical work \cite{Migliorati_1,Migliorati_2, Zhou, Cohen,narayan_christoffel_2014} has centered around determining the stability conditions for least squares, when using \emph{independent and identically distributed} (iid) or random sampling. In Cohen et al. \cite{Cohen} the authors analyze univariate polynomials and observe for approximating an m-$th$ order polynomial approximately $cm^2$ points are required when sampling randomly from a uniform distribution on the inputs, or approximately $cm \textrm{log}(m)$ points when sampling from a Chebyshev distribution, for some constant $c$. Theoretical extensions to multivariate polynomial spaces can be found in Chkifa et al. \cite{Chkifa}, and with applications to multivariate orthogonal polynomial expansions in tensor-product and total-order index spaces in \cite{Zhou}. 

In this paper we introduce a new approach for deterministically subsampling quadratures in the context of least squares approximations. Given a polynomial subspace and a tensor product quadrature grid, the central idea is to select a number of subsamples from the grid equal to the dimension of the subspace using a QR factorization column pivoting strategy. Further pruning of the polynomial subspace is performed via heuristics. We remark here that good performance of our method depends on the pruning strategy adopted, and in general it is difficult to develop rigorous bounds.

Details of the approach are in section~\ref{sec:algo} with a discussion in section~\ref{sec:discussion}. This is followed by numerical examples in sections~\ref{sec:examples} and ~\ref{sec:numex}. All our results, and code to produce them, can be found at: www.effective-quadratures.org/papers. These codes use routines from our effective-quadratures toolkit \cite{Seshadri2017}. 

\subsection{Preliminaries \& notation}
Let $\bm{\zeta}=(\zeta^{(1)}, \ldots, \zeta^{(d)} ) $ be a $d$-dimensional vector of mutually independent random variables with joint probability density $\bm{\rho}$ and marginal densities $\rho_i$ related by $\bm{\rho}(\bm{\zeta})=\prod_{i=1}^d \rho_i\left(\zeta^{(i)}\right)$ defined on $\mathbb{R}^d$.
\subsection{On polynomials}\label{sec:polynomials}
Let $\left\{ \psi_j^{(i)} \right\}_{j=0}^\infty$ denote a family of polynomials $L^2$-orthogonal on $\mathbb{R}$ when weighted by the density $\rho_i$:
\begin{align}
  \int_\mathbb{R} \psi_i^{(k)}(s) \psi_j^{(k)}(s) \rho_k(s) ds = \mathbb{E} \left[ \psi_i^{(k)} \left(\zeta^{(k)}\right) \psi_j^{(k)} \left(\zeta^{(k)}\right) \right] = \delta_{i,j},
\end{align}
where $\delta_{i,j}$ is the Kronecker delta. Existence of such a family is ensured under mild assumptions on $\rho_i$ \cite{gautschi_orthogonal_2004}; the $\rho_i$-weighted $L^2$-completeness of the polynomial family can be established under some additional technical assumptions \cite{lubinsky_survey_2007,ernst_convergence_2012}.

A multivariate polynomial $\bm{\psi_j}: \mathbb{R}^d \rightarrow \mathbb{R}$ can be defined as a product of univariate polynomials,
\begin{align}
  \bm{\psi_{j}}\left(\bm{\zeta}\right)&=\prod_{k=1}^{d}\psi^{(k)}_{j_{k}}\left(\zeta^{(k)}\right), & \bm{j} &= \left(j_1, \ldots, j_d \right) \in \mathbb{N}_0^d,
\end{align}
where $\bm{j}$ is a multi-index that denotes the order (degree) of $\bm{\psi_j}$ and its composite univariate polynomials $\psi^{(k)}_{j_k}$. The family $\left\{\bm{\psi_j} \right\}_{\bm{j} \in \mathbb{N}_0^d}$ defined in this way is mutually orthogonal in $L^2$ weighted by $\bm{\rho}$.

For computational purposes, we require a finite number of polynomials $\bm{\psi_j}$. This finite set is chosen by restricting the multi-index $\bm{j}$ to lie in a finite multi-index set $\mathcal{J}$.  There are four well-known multi-index sets $\mathcal{J}$ that have proven fruitful for parametric approximation: tensor product index sets, total order index sets, hyperbolic cross spaces \cite{hyperbolic_cross, Sparse2} and hyperbolic index sets \cite{Sudret_hyperbolic}. Each of these index sets in $d$ dimensions is well-defined given a fixed $k \in \mathbb{N}_0$, which indicates the maximum polynomial degree associated to these sets. 

Isotropic tensor product index sets consist of multi-indices satisfying $\max_k j_k \leq k$, and have a cardinality (number of elements) equal to $(k+1)^d$. Total order index sets contain multi-indices satisfying $\sum_{i=1}^d j_{i} \leq k$. Loosely speaking, total order indices disregard some higher order interactions between dimensions present in tensorized spaces, and a total order index set $\mathcal{J}$ has cardinality
\begin{equation}
\left| \mathcal{J} \right|=\left(\begin{array}{c}
k+d\\
k
\end{array}\right). 
\end{equation}
Hyperbolic cross sets contain indices governed by the rule $\prod_{i=1}^d\left(j_{i}+1\right)\leq k+1$, and prune even more tensorial interaction terms than total order index spaces. The cardinality of this last index set is approximately $(k+1)(1+\log(k+1))^{d-1}$ \cite{Zhou}. Finally, a hyperbolic index set contain indices that satisfy
\begin{equation}
\left(\sum_{i=1}^{d}j_{i}^{q}\right)^{1/q} \leq k,
\label{eq:hyperbolic-space}
\end{equation}
where $q$ is a user-defined constant that can be varied from 0.2 to 1.0. When $q=1$ the hyperbolic index space is equivalent to a total order index space, while for values less than unity higher-order interactions terms are eliminated \cite{Sudret_hyperbolic}. In the numerical examples in this paper we will form approximations from the hyperbolic index space (not the hyperbolic cross set). 

\subsection{On quadrature rules}\label{sec:quadrature}
We assume existence of a quadrature rule $\left\{ \left( \boldsymbol{\zeta}_i, \omega^2_i \right) \right\}_{i=1}^m$, with non-negative weights $\omega_i^2$, such that
\begin{align}\label{eq:quadrature-accuracy}
  \sum_{i=1}^m \omega^2_i \boldsymbol{\phi}_{\boldsymbol{j}}\left( \boldsymbol{\zeta}_i \right) \boldsymbol{\phi}_{\boldsymbol{\ell}}\left( \boldsymbol{\zeta}_i \right) = \int_{\mathbb{R}^d} \boldsymbol{\phi}_{\boldsymbol{j}}\left( \boldsymbol{\zeta}\right) \boldsymbol{\phi}_{\boldsymbol{\ell}}\left( \boldsymbol{\zeta}\right) \boldsymbol{\rho}(\boldsymbol{\zeta}) dx{\boldsymbol{\zeta}} &= \delta_{\boldsymbol{\ell}, \boldsymbol{j}}, & \boldsymbol{j}, \boldsymbol{\ell} &\in \mathcal{J}.
\end{align}
Thus the choice of quadrature rule is intimately tied to the choice of $\mathcal{J}$. Ideally, the cardinality $m$ of this quadrature rule should be as small as possible. In one dimension, we can achieve equality in the above expression with $m = \left|\mathcal{J}\right|$ for essentially any density $\rho$ by using a \textit{Gaussian} quadrature rule \cite{freud_orthogonal_1971}. However, rules of this optimal (smallest) cardinality are not known in the multivariate setting for general $\boldsymbol{\rho}$ and $\mathcal{J}$. The construction of quasi-optimal multivariate quadrature rules, even over canonical domains, is a challenging computational problem \cite{taylor_cardinal_2007} and no universal solutions are currently known.

On tensor-product domains (as is assumed in this paper) one quadrature rule satisfying \eqref{eq:quadrature-accuracy} can be constructed by tensorizing univariate rules. For example, \eqref{eq:quadrature-accuracy} holds if the univariate rules are Gaussian quadrature rules with sufficiently high accuracy. Since the cardinality of the resulting tensorial rule grows exponentially with dimension $d$, evaluating the model over a full tensor-product quadrature rule quickly becomes infeasible. In this paper, we will use a tensorial Gauss quadrature rule with high enough accuracy to ensure \eqref{eq:quadrature-accuracy}, but in principle it is also reasonable to apply our approach when the first equality in \eqref{eq:quadrature-accuracy} is only approximate.

This motivates the goal of this paper: prune a full tensorized quadrature rule $\left\{ \left( \boldsymbol{\zeta}_i, \omega^2_i \right) \right\}_{i=1}^m$ via a subsampling procedure so that the subsampled grid has an acceptable cardinality while hopefully maintaining the accuracy properties associated to an approximation using the index set $\mathcal{J}$.

\subsection{On matrices}
Matrices in this paper are denoted by upper case bold letters, while vectors are denoted by lower case bold letters. For a matrix, $\mD \in \mathbb{R}^{m \times n}$ with $m \geq n$, singular values are defined by $\sigma_i\left(\mD \right)$, with $i=1,\ldots, n$. Unless explicity stated otherwise, the singular values are arranged in descending order, i.e, $\sigma_1\left(\mD \right) \geq \ldots \geq \sigma_n\left(\mD \right)$. The $\ell^2$ condition number of $\mD$ is denoted $\kappa\left(\mD \right)$ and is the ratio of the largest singular value to the smallest. The singular values of $\mD$ coincide with those of $\mD^T$.

\section{Effectively subsampled quadratures}
\label{sec:algo}
In this section we describe polynomial least squares via \emph{effectively subsampled quadratures}. To aid our discussion, the overall strategy is captured in Figure~\ref{figureimp}; we describe the details below.

\begin{figure}
\centering
\includegraphics{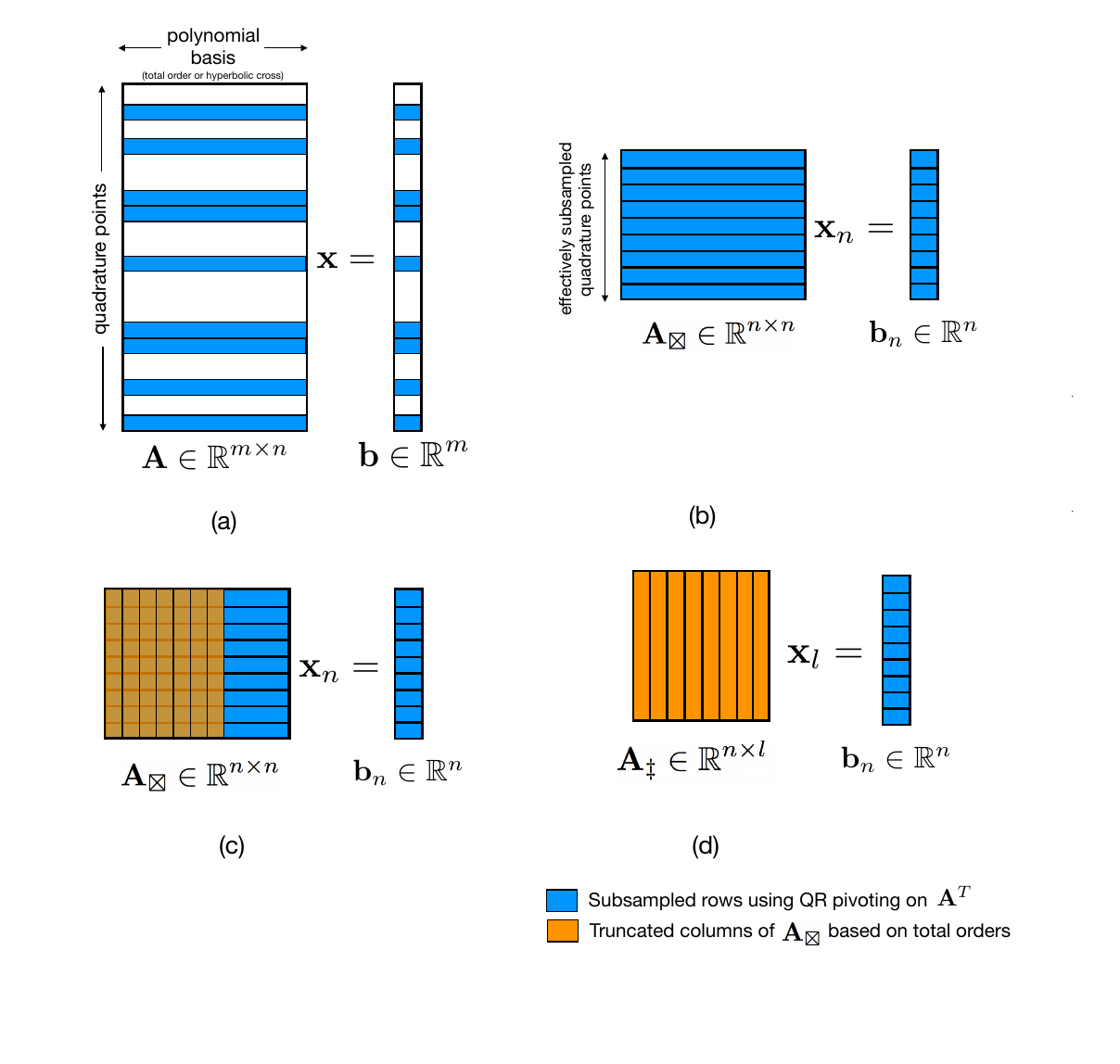}
\caption{Step-by-step outline of the effectively subsampled quadrature technique for computing polynomial least squares approximations: (a) Setting up the $\mA$ matrix; (b) QR factorization with column pivoting on $\mA^T$; (c) Column pruning; (d) Solving the least squares problem.}
\label{figureimp}
\end{figure}

\subsection{Setting up the $\mathbf{A}$ matrix}
Let $\mA \in \mathbb{R}^{m \times n}$ be a matrix formed by evaluating $n$ multivariate orthonormal polynomials (constructed as described in Section \ref{sec:polynomials}) at a tensor grid formed of $m$ quadrature points (described in Section \ref{sec:quadrature}). Individual entries of $\mA$ are given by
\begin{align}
  \mA\left(i,\boldsymbol{j}\right)&=\omega_{i}\bm{\psi}_{}\left(\bm{\zeta}_{i}\right), & \boldsymbol{j}\in\mathcal{J}, \hskip 5pt i=1, \ldots, m.
\end{align}
where in this paper we implicitly assume that the index set $\mathcal{J}$ is a hyperbolic\footnote{Either a hyperbolic cross space or a hyperbolic space} or total order index set with the condition that $\left|\mathcal{J}\right|=n \ll m$. We assume the quadrature rule is accurate enough so that \eqref{eq:quadrature-accuracy} holds. We consider $\mA$ a matrix by assuming a linear ordering of the elements in $\mathcal{J}$; the particular ordering chosen may be arbitrary in the context of this paper.

We define $\vb \in \mathbb{R}^{m}$ to be the vector of weighted model evaluations at the $m$ quadrature points, where individual entries of $\vb$ are given by
\begin{equation}
\vb(i)=\omega_{i}f\left(\bm{\zeta}_i\right),
\label{2}
\end{equation}
where $f\left(\cdot\right)$ represents the quantity of interest from our computational model; we seek to approximate this with $g(s)$
\begin{equation}
f\left(s\right)\approx g(s) = \sum_{i\in\mathcal{J}}^{n}x_{i} \bm{\psi_{i}}\left(s\right).
\label{main_poly_equ}
\end{equation}
This is equivalent to solving the least squares problem $\mA \vx = \vb$ for the coefficients $\vx \in \mathbb{R}^{n}$. This requires our model to be evaluated at each tensor grid quadrature point. A non-deterministic (i.e., randomized) approach to reduce the cost associated with this least squares problem is to randomly subsample tensor grid quadrature points as outlined in \cite{Zhou}. This strategy requires a reduction in the number of basis terms in $\mA$ to promote stability. We make specific comments regarding the randomized strategy in the numerical studies section of this paper. To contrast, our approach in this paper is deterministic and rooted in a heuristic that is tailored for least squares problems. 

\subsection{QR factorization with column pivoting on $\mathbf{A}^T$}
To reduce the cost associated with evaluating our model at each point in a tensor grid, we utilize QR column pivoting---a well-known heuristic for solving rank deficient least squares problems. QR column pivoting works by (i) determining the numerical rank $r<n$ of an $m$-by-$n$ matrix, and (ii) permuting columns of the matrix such that the first $r$ columns are linearly independent \cite{Hansen}. Here, we apply this heuristic for subselecting rows of $\mA$. Let the QR column pivoting factorization of $\mA^T$ be given by
\begin{equation}
\mA^{T}\mP=\mQ\left(\begin{array}{cc}
\mR_{1} & \mR_{2}\end{array}\right)
\label{equ_QR}
\end{equation}
where $\mQ \in \mathbb{R}^{n \times n}$ is orthogonal, $\mR_{1} \in \mathbb{R}^{n \times n}$ is nonsingular\footnote{By assuming \eqref{eq:quadrature-accuracy}, $\boldsymbol{A}$ must be of full rank, equal to $\min \{m,n\} = n$. Thus $\boldsymbol{R}_1$ also has rank $n$.} and upper triangular with positive diagonal elements, and $\mR_{2} \in \mathbb{R}^{n\times (m-n)}$. The matrix $\mP$ is a \emph{permutation matrix} that permutes columns of $\mA^T$ such that the diagonal entries of $\mR_{1}$ are greedily maximized. It should be noted that the factorization in~\eqref{equ_QR} is not necessarily unique as there may be numerous ways to select $n$ linearly independent columns of $\mA$ \cite{Bjorck}. To determine precisely which rows of $\mA$ to subselect (equivalent to determining which points to subsample) we define a vector $\bm{\pi}$ that converts the pivots encoded in the matrix $\mP$ to the specific rows in $\mA$ by
\begin{equation}
\bm{\pi}=\mP^{T}\vu,
\label{4}
\end{equation}
where $\vu = (1,2,\ldots,m)^T$ is a vector of integers from 1 to $m$. The vector $\boldsymbol{\pi}$ contains ordered rows of $\boldsymbol{A}$ that are subselected via the $Q R$ factorization. For clarity let $\bm{\pi}_n = \bm{\pi}(1:n)$ be the first $n$ entries of $\bm{\pi}$, and the operator $\mathcal{G}_{\bm{\pi}_{n}}$ that selects the rows indexed by $\bm{\pi}_n$. We define
\begin{equation}
\mA_{\boxtimes} = \mathcal{G}_{\bm{\pi}_{n}}\left(\mA\right)
\label{5}
\end{equation}
where $\mA_{\boxtimes} \in \mathbb{R}^{n \times n}$ (see Figure \ref{figureimp}(b)). For ease in notation we define column-wise vectors of $\mA^T$ using $\vc$ as follows:
\begin{equation}
\mA^{T}=\left(\begin{array}{ccc}
- & \va_{1}^{T} & -\\
 & \vdots\\
- & \va_{n}^{T} & -
\end{array}\right)=\left(\begin{array}{ccc}
| &  & |\\
\vc_{1} & \ldots & \vc_{n}\\
| &  & |
\end{array}\right).
\end{equation}

The QR pivoting algorithm used in this study---shown in Algorithm 1---is based on the work of Dax \cite{Dax} and uses modified Gram-Schmidt QR factorization. As shown in lines 14-15 of the algorithm, the vector $\vc_j$ is orthogonalized by iteratively projecting it on to the subspace orthogonal to \textsf{span}$(\vg_1, \ldots, \vg_{k-1})$, where $\vg_k=\vc_k / \left\Vert \vc_{k}\right\Vert _{2}$ \cite{Hansen}. 
\begin{algorithm}
  \caption{QR with column pivoting on $\mA^{T}=\left(\vc_{1},\ldots,\vc_{m}\right)\in \mathbb{R}^{n \times m}$ where $m > n$}\label{euclid}
  \begin{algorithmic}[1]    
    \Procedure{QR Column Pivoting}{$\mA^T$}
    \State Compute and store column norms of $\mA^T$ in \texttt{colnorms} 
    \State Declare permutation vector \texttt{$\pi$=1:m}
      \For{ \texttt{k=1:n} } 
      	\State Store \texttt{max(colnorms(k:n))} index as \texttt{jmax}
      	\State Set \texttt{jmax = jmax + (k-1)}
      	
      	\If{$k$ is not equal to jmax}
      		\State Swap $\vc_{k}$ with $\vc_{jmax}$ \Comment{Swaping}
      		\State Swap \texttt{colnorms(k)} with \texttt{colnorms(jmax)} 
      		\State Swap \texttt{$\pi$(k)} with \texttt{$\pi$(jmax)}
      	\EndIf
      	
      	\If{$k$ is not equal to n} \Comment{Orthogonalization}
        	\For{\texttt{j=k+1:n}}
        		\State $\vg= \vc_k / \left\Vert \vc_{k}\right\Vert _{2}$
        		\State $\vc_j = \vc_j - \vg^T \vc_j \vg$
        		\State Update \texttt{colnorms(j)}
        	\EndFor
        \EndIf	

        \If{$k$ is not equal to 1} \Comment{Reorthogonalization}
    		\For{ \texttt{i=1:k-1} }
    		\State $\vh= \vc_i / \left\Vert \vc_{i}\right\Vert _{2}$
    		 \State $\vc_k = \vc_k - \vh^T \vc_k \vh$ 
    		 \EndFor 
    	\EndIf 
      \EndFor
      \State \textbf{return} \texttt{$\pi$}
    \EndProcedure
  \end{algorithmic}
\end{algorithm}
Householder QR pivoting (see page 227 of \cite{Golub_book}) may also be used instead. Both algorithms require an initial computation of the column norms followed by subsequent updates after iterative orthogonalization. As Bj{\"o}rck \cite{Bjorck} notes, once the initial column norms have been computed, they may be updated using the identity
\begin{equation} \label{equ_rule}
\begin{split}
\left\Vert \vc_{j}^{\left(k+1\right)}\right\Vert _{2} & = \sqrt{\left\Vert \va_{k}^{(k)}\right\Vert _{2}^{2}-\left(\vg^{T}\vc_{j}^{\left(k\right)}\right)^{2}} \\
 & = \left\Vert \vc_{k}^{(k)}\right\Vert _{2}\left(1-\left(\frac{\vg^{T}\vc_{j}^{(k)}}{\left\Vert \vc_{k}^{(k)}\right\Vert _{2}}\right)^{2}\right)^{1/2},
\end{split}
\end{equation}
instead of directly computing the column norms. The above identity is applied in line 16 of Algorithm 1. While~\eqref{equ_rule} reduces the overhead of pivoting from $\mathcal{O}\left(mn^2\right)$ flops to $\mathcal{O}\left(mn\right)$ flops \cite{Golub_book}, there are some sailent computational issues that need to be considered when using such an updating rule. (See page 6 of \cite{Dax}.) A small but notable difference in our implementation of the QR with column pivoting algorithm is that our for-loop terminates at $\texttt{k=n}$, which is sufficient for computing the $n$ pivots we need. A python implementation of the above algorithm is included in our \texttt{effective-quadratures} toolkit, and is utilized in all the numerical studies in this paper. For QR factorizations with column pivoting, there are comprehensive \texttt{LAPACK} routines, such as \texttt{DGEQP3} \cite{Greg}.

\subsubsection{Relation to rank revealing QR factorizations}
The factorization in~\eqref{equ_QR} is called a \emph{rank revealing QR (RRQR)} factorization if it satisfies the property that $\sigma_{min}(\mR_1) \geq \sigma_{k}(\mA^T)/p(n,m)$, where $p(n,m)$ is a function bounded by a lower order polynomial in $n$ and $m$. The Businger and Golub Householder QR with pivoting algorithm \cite{Golub_book}, the Chandrasekaran and Ipsen algorithm \cite{Ipsen} and that of Golub, Klema and Stewart \cite{GKS} fall into this category. Bounds on all $n$ singular values of $\mR_1$---and not just the minimum---can be obtained when using algorithms that yield a \emph{strong RRQR} factorization; a term coined by Gu and Eisenstat \cite{Gu}. The latter also provide an efficient algorithm for computing a strong RRQR. Broadbent et al. \cite{Broadbent} follow this work and proved that a strong RRQR applied to matrices with more columns than rows---as in the case for our QR factorization with $\mA^T$---yields the following identity
\begin{equation}
\sigma_{i}\left(\mR_{1}\right)\geq\frac{\sigma_{i}\left(\mA\right)}{\sqrt{1+\delta^{2}n\left(m-n\right)}} \; \;  \; \textrm{for} \; \; \; 1 \leq i \leq n,
\label{broadbent}
\end{equation}
for some constant $\delta > 1$. As the singular values of $\mR_1$ are equivalent to those of $\mA_\boxtimes$, this imples that singular values $\sigma_i(\mA_\boxtimes)$ lie between $\sigma_i(\mA)$ and the right-hand-side of~\eqref{broadbent}. We briefly analyze the stability of solving a linear system involving the rank-$n$ matrix $\boldsymbol{A}_\boxtimes$ by bounding its condition number relative to the condition number.
\begin{lemma}\label{lemma_1}
With $\delta$ the parameter in the inequality \eqref{broadbent}, and $\kappa(\cdot)$ the 2-norm condition number of a matrix, then
  \begin{align}\label{eq:conditioning-lemma-result}
    \kappa\left(\boldsymbol{A}_{\boxtimes} \right) \leq \kappa\left(\boldsymbol{A} \right) \sqrt{ 1 + \delta^2 n (m-n) } 
  \end{align}
\end{lemma}
\begin{proof}
We first note that 
  \begin{subequations}\label{eq:conditioning-lemma}
  \begin{align}\label{eq:conditioning-lemma-a}
    \kappa\left(\boldsymbol{A} \right) \coloneqq \frac{\sigma_1\left(\boldsymbol{A}\right)}{\sigma_n\left(\boldsymbol{A}\right)}.
  \end{align}
  By \eqref{broadbent}, we have
  \begin{align}\label{eq:conditioning-lemma-b}
    \frac{1}{\sigma_n\left(\boldsymbol{A}_\boxtimes\right)} \leq \frac{\sqrt{ 1 + \delta^2 n (m-n) }}{\sigma_n\left(\boldsymbol{A}\right)}
  \end{align}
  Finally, since $\boldsymbol{A}_\boxtimes$ is precisely a submatrix of $\boldsymbol{A}$ (see \eqref{5}), then the singular values of $\boldsymbol{A}$ and $\boldsymbol{A}_\boxtimes$ interlace \cite{thompson_principal_1972}, in particular,
  \begin{align}\label{eq:conditioning-lemma-c}
    \sigma_1\left(\boldsymbol{A}_\boxtimes\right) \leq \sigma_1\left(\boldsymbol{A}\right).
  \end{align}
  Combining the three relations \eqref{eq:conditioning-lemma} with $\kappa\left(\boldsymbol{A}_\boxtimes\right) \coloneqq \sigma_1\left(\boldsymbol{A}_\boxtimes\right)\, / \, \sigma_n\left(\boldsymbol{A}_\boxtimes\right)$ proves \eqref{eq:conditioning-lemma-result}.
  \end{subequations}
\end{proof}

\subsubsection{Relation to subset selection}
Like QR with column pivoting, \emph{subset selection} is an alternative heuristic that aims to produce a well-conditioned submatrix with linearly independent columns. In practice, subset selection can produce a submatrix with smaller condition number than that provided by QR with column pivoting \cite{Golub_book}. The algorithm has two key steps that can be adapted to determine which rows of $\mA$ to subselect for $\mA_\boxtimes$. The first step involves computing the singular value decomposition of $\mA^T$. The next step requires QR column pivoting to be applied to a matrix formed by the transpose of the first $n$ right-singular vectors of $\boldsymbol{A}$,
\begin{equation}
\mV\left(:,1:n\right)^{T}\mP=\mQ\mR,
\end{equation}
where as before the columns of $\mP$ encode the permutations. Equations~\eqref{4} and~\eqref{5} can subsequently be used determine $\mA_{\boxtimes}$. One of main computational bottlenecks with \emph{subset selection} is the aggregated cost of performing both an SVD---costing $\mathcal{O}(m^2n^3)$ flops---with a QR column pivoting routine \cite{Golub_book}.

\subsection{Column pruning}
In~\eqref{2} we defined elements of the vector $\vb$ to be the weighted model evaluations at all $m$ quadrature points. In practice we only require the model evaluations at the $n$ quadrature points identified by $\bm{\pi}_n$. We define the \emph{effectively subsampled quadrature} points and weights to be
\begin{equation}
\bm{\zeta}_{e,j}=\zeta_{\bm{\pi}\left(j\right)}, \; \; \; \; \; \; \bm{\omega}_{e,j}=\omega_{\bm{\pi}\left(j\right)}
\label{pts_wts}
\end{equation}
respectively, for $j=1,\ldots, n$, where $\bm{\pi}$ is the QR permutation vector defined in~\eqref{5}. Thus the vector of weighted model evaluations at these points is given by 
\begin{equation}
\vb_{n}\left(j\right)=\bm{\omega}_{e,j}f\left(\bm{\zeta}_{e,j}\right), \; \; \; \; \; \; \; j = 1, \ldots, n.
\end{equation}
Assembling the \emph{square} linear system of equations yields
\begin{equation}
\underset{\vx_{n}}{\textrm{minimize}}\left\Vert \mA_\boxtimes \vx_{n}-\vb_{n}\right\Vert _{2}.
\label{lsq}
\end{equation}
The subscript $n$ in $\vx_n$ simply denotes the number of coefficient terms---equivalent to the cardinality of the polynomial basis defined by $\mathcal{J}$---that are to be solved for.  While \eqref{lsq} can be solved to yield accurate coefficient estimates for smooth functions, it is generally ill-advised, as we wish to \emph{approximate} rather than \emph{interpolate}. Consequently, we prune down the number of columns of $\mA_{\boxtimes}$ from $n$ to $l$. It is difficult to offer heuristics for column pruning as this will no doubt vary across applications, as it is dependent on which polynomial basis terms can be truncated without significant loss of accuracy in approximating $f$. Our experience (from the examples later in this paper) suggests that better results are obtaining by eliminating columns with the highest total degrees first, and our results using $n/l \in [1,1.5]$ show promise.

The matrices $\mA$ and $\mA_{\boxtimes}$ are formed from the polynomial basis defined by $\mathcal{J}$. The above procedure prunes elements from $\mathcal{J}$. We define the index set that results from this pruning as $\mathcal{I}$. I.e., $\mathcal{I}$ is defined by
\begin{equation}
  \left|\mathcal{I}\right|=l \; \; \; \; \mathcal{I}\subseteq\mathcal{J}, \; \; \; \; \boldsymbol{k}\in\mathcal{J}\setminus\mathcal{I}\implies\sum_{i=1}^{d}k_{d}\geq\sum_{i=1}^{d}j_{d}\; \; \; \textrm{for} \; \textrm{all} \; \; \bm{j} \in \mathcal{I}.
\end{equation}
This does not uniquely define $\mathcal{I}$ as there is usually not a unique element of $\mathcal{J}$ with highest total order. In this paper we perform the following methodology for pruning a single element from $\mathcal{J}$: We specify an (arbitrary) ordering of elements in $\mathcal{J}$, and based on this ordering prune the first $\boldsymbol{k} \in \mathcal{J}$ achieving the maximum total order.

Let $\mA_{\ddagger} \in \mathbb{R}^{n \times l}$ be the submatrix of $\mA_{\boxtimes}$ associated with the pruned set, $\mathcal{I}$, i.e., 
\begin{equation}
\mA_{\ddagger}\left(i, \bm{j}\right)=\bm{\omega}_{e,i}\bm{\psi}_{\bm{j}}\left(\bm{\zeta}_{e,i}\right), \; \; \; \; \bm{j}\in\mathcal{I}, \; \; i=1,\ldots,n.
\end{equation}

Regardless of how pruning is performed, the following result holds:
\begin{corollary} 
With $\delta$ as in Lemma \ref{lemma_1}, let $\boldsymbol{A}_{\ddagger}$ be the column-pruned version of $\boldsymbol{A}_\boxtimes$. Then 
\begin{equation}\label{eq:corollary-result}
\kappa\left(\mA_{\ddagger}\right) \leq \kappa\left(\mA_{\boxtimes}\right) \leq \kappa\left(\boldsymbol{A}_n\right) \sqrt{ 1 + \delta^2 n (m-n) } 
\end{equation}
\end{corollary}

\begin{proof}
From the interlacing property of singular values
we have
\begin{equation}
\sigma_1(\mA_{\boxtimes}) \geq \sigma_1(\mA_{\ddagger}) \; \; \textrm{and} \; \;  \sigma_l(\mA_{\ddagger}) \geq \sigma_n(\mA_{\boxtimes}),
\end{equation}
where the singular values are ordered such that $\sigma_{1}\left(\cdot\right) \geq \sigma_{2}\left(\cdot\right) \geq \ldots \geq \sigma_{n}\left(\cdot\right)$. Then 
\begin{equation}
\frac{\sigma_{1}\left(\mA_{\ddagger}\right)}{\sigma_{l}\left(\mA_{\ddagger}\right)} \leq \frac{\sigma_{1}\left(\mA_{\boxtimes}\right)}{\sigma_{n}\left(\mA_{\boxtimes}\right)}
\implies\kappa\left(\mA_{\ddagger}\right)\leq\kappa\left(\mA_{\boxtimes}\right).
\end{equation}
The second inequality in \eqref{eq:corollary-result} is an application of Lemma \ref{lemma_1}.
\end{proof}
This implies that reducing the number of columns---in particular for our case of eliminating the basis terms with the highest total degrees---will not cause an increase in the condition number of $\mA_{\ddagger}$ compared to $\mA_{\boxtimes}$. Thus compared to solving a least-squares problem with the best rank-$n$ approximation to $\boldsymbol{A}$, we suffer a penalty of the order $\delta \sqrt{m n}$ (when $m \gg n$).


\subsection{Solving the least squares problem}
We now reach the main objective of this paper, which is to solve the least squares problem given by
\begin{equation}
\underset{\vx_{l}}{\textrm{minimize}}\left\Vert \mA_{\ddagger} \vx_{l}-\vb_{n}\right\Vert _{2},
\label{lsq_main}
\end{equation}
for the coefficients $\vx_{l} \in \mathbb{R}^{l}$. There are two sailent points we wish to emphasize upon when solving~\eqref{lsq_main}. The first pertains to preconditioning. We impose a unit length diagonal column scaling, a preconditioner that is frequently selected when solving least squares problems \cite{Hansen}. We define the preconditioner (a nonsingular matrix) $\mS \in \mathbb{R}^{n \times n}$ as
\begin{equation}
\mS=\left(\begin{array}{ccc}
\left\Vert \va_{\ddagger1}\right\Vert _{2}\\
 & \ddots\\
 &  & \left\Vert \va_{\ddagger n}\right\Vert _{2}
\end{array}\right) 
\end{equation}
where vectors $\left\{ \va_{\ddagger 1}, \ldots , \va_{\ddagger n} \right\}$ are the columns of $\mA_{\ddagger}$. This yields the modified least squares problem
\begin{equation}
\underset{\vz_{l}}{\textrm{minimize}}\left\Vert \mA_{\ddagger} \mS^{-1}\vz_{l}-\vb_{n}\right\Vert _{2} \; \; \; \textrm{with} \; \; \; \mS \vx_{l} = \vz_{l}.
\label{lsq_main_precond}
\end{equation}
This brings us to a well-known observation about solving least-squares problems: Solving the normal equations, e.g., computing the inverse of $\mS^{-T}\mA_{\ddagger}^{T}\mA_{\ddagger}\mS^{-1}$ or even forming an associated Cholesky factorization, is relatively unstable. Our results reported here solve the system \eqref{lsq_main_precond} via the more stable $QR$ factorization approach (see Chapter 4 of \cite{Hansen}).

\subsection{A step-by-step algorithm}
\label{sec:approach}
We summarize this section with an algorithm incorporating the previously defined computational techniques given a function $f=f\left(\bm{\zeta}\right)$ defined on $\mathcal{R} = \mathrm{supp}\left(\rho\right)$ with joint probability density function $\bm{\rho}\left(\bm{\zeta}\right)$. 
\begin{enumerate}
\item \textbf{Polynomial basis selection:} Let a joint density $\bm{\rho(\zeta)}$ be given. Select a hyperbolic or total order index set $\mathcal{J}$, with a cardinality $\left|\mathcal{J}\right|=n$.
\item \textbf{Initial grid selection:} Choose a tensorized quadrature rule consisting of $m$ points and weights $\left\{ \zeta_{j},\omega_{j}\right\} _{j=1}^{m}$ such that the quadrature rule satisfies~\eqref{eq:quadrature-accuracy}. This defines the matrix $\mA$. 
\item \textbf{Subselecting points from initial grid:} Compute the QR with pivoting factorization $\mA^{T} \mP = \mQ \mR$ and select the first $n$ entries encoded in $\mP$ to compute the matrix $\mA_{\boxtimes}$. Evaluate the model at the quadrature points corresponding to these $n$ pivots to compute $\vb_n$.
\item \textbf{Column pruning:} For noisy $f$, prune down the number of columns by eliminating the columns of $\mA_{\boxtimes}$ that correspond to the highest total orders. Store the remaining columns in a new matrix, $\mA_{\ddagger}$. As a heuristic, we recommend pruning down by ratios of 1.25 and 1.50.
\item \textbf{Least squares:} Solve the least squares problem with $\mA_{\ddagger}$ and $\vb_n$.
\end{enumerate}

\section{Discussion \& heuristics}
\label{sec:discussion}
The stategy described is similar to the procedure of constructing \emph{approximate Fekete points \cite{AFP}}; this latter approach has been used in the context of finding near-optimal interpolation points in multidimensional space. An alternative way to think about the strategy above is that for a given design matrix $\mA$, and a maximum number of permitted model evaluations $n$, we are extracting a set of at most $n$ sample points. Our algorithm in this context offers a deterministic recipe for subsampling a tensor grid. The notable difference between our algorithm and that produced by approximate Fekete points is that we introduce (square-root) quadrature weights $\omega_i$ in the definition of $\mA$. The algorithm as we have decribed it can produce accurate low-rank approximations to the matrix $\mA$ \cite{harbrecht_2012}. 

We note that we have described this quadrature subselection strategy as an attempt to construct a well-conditioned design matrix. Remarkably, if $\omega_i \equiv 1$, then the limiting behavior of the points selected via this algorithm is known. Consider the univariate case on a bounded set but with large $m$ and $n$. In this case it is known that as $m$ and $n$ tend to infinity appropriately, the QR selection strategy chooses points that distribute according to the Chebyshev (arcsine) measure \cite{bos_2011}. In addition, this property holds in the multidimensional setting on a hypercube (the set formed from the Cartesian product of univariate bounded intervals). If we use the QR strategy to select points from a sufficiently dense grid on a hypercube, these points distribute according to the product Chebyshev measure on the hypercube. Precise conditions on $m$, $n$, the type of grid, and the type of convergence is given in \cite{bos_2011}. 

\subsection{Memory requirements}
One of the disadvantages of QR with column pivoting is that to pivot the column with the largest norm, we must compute the norms of all the remaining columns. In the next stage of the for-loop, we then need to downdate these norms. This is where Identity~\eqref{equ_rule} is useful. Occasionally, a cancellation occurs in~\eqref{equ_rule} requiring knowledge of all the column entries to recompute its norm. Thus, it is not possible to carry out QR with column pivoting by storing 2-3 rows or columns at a time; access to the full matrix is required. This implies that the cost of our technique scales exponentially with the dimension of $f$'s inputs---owing to QR column pivoting. This is one drawback of the current approach. This scaling is however independent from the number of evaluations of $f$ that we require. 

While writing the entire matrix $\mA$ on the disk and then extracting rows and columns as required is one option, it is far from elegant. One possible path forward lies in randomized QR column pivoting techniques \cite{RandomQR}. These techniques are promising because they restrict the size of $\mA$; however, a detailed investigation of this procedure is outisde the scope of this manuscript.

\section{An Analytical Example}
\label{sec:examples}
In this simple analytical example we set 
\begin{equation}
f(\bm{\zeta}) = exp(\zeta^{(1)} + \zeta^{(2)})
\end{equation}
defined over $\mathcal{R}=[-1,1]^2$ with $\bm{\rho(\zeta)}$ the uniform density. We wish to approximate $f$ using a basis of Legendre orthonormal polynomials with \emph{effectively subsampled quadratures}. We subsample an isotropic tensor grid formed from a 21-point Gauss-Legendre stencil in each dimension. We let $\mathcal{J}$ be a total order basis for our least squares computations. In this example, we compare our method with the randomized sampling approach of \cite{Zhou} and investigate the sensitivity of the procedure to column pruning. 

Figure~\ref{problem_1b}(a) plots the approximation errors in the polynomial coefficients using both the randomized and effectively subsampled methods. Here values on the $x-$axis represent the maximum degree defining the total-order index set $\mathcal{J}$. The $y-$axis defines the coefficient error on a base-10 logarithmic scale. For each $x-$axis value, this error is computed using
\begin{equation}
\epsilon=\left\Vert \vx_{\otimes,\mathcal{J}}- \vx_{n}\right\Vert _{2}
\label{error1}
\end{equation}
where $\vx_{\otimes}$ are the coefficients estimated from a 21-point tensor grid quadrature rule. The subscript $\mathcal{J}$ in $\vx_{\otimes,\mathcal{J}}$ denotes the coefficient values only associated with the multi-indices in the total order index set $\mathcal{J}$. As mentioned earlier, the coefficients $\vx_n$ (see Figure~\ref{figureimp}) are obtained by solving the least squares problem with the polynomial basis $\mathcal{J}$. The coefficient errors in Figure~\ref{problem_1b}(a) correspond to the errors associated with solving the least squares problem on the square matrix $\mA_{\boxtimes}$. The green line in the figure shows the coefficient error resulting from using \emph{effectively subsampled quadratures}, while the yellow shaded regions denote the minimum and maximum values of $\epsilon$ obtained using randomized subsampling---with 20 repetitions of the experiment. The red line represents the mean result from those 20 trials. 

In Figures~\ref{problem_1b}(b-d), for a fixed maximum degree $k$, we prune down the number of columns in $\mA_{\boxtimes}$ from $n=\left| \mathcal{J} \right|  $ to $l$ to yield a smaller index set $\mathcal{I}$ and matrix $\mA_{\ddagger}$. This \emph{column pruning} starts with those columns that have the highest total orders. The errors we report here are given by 
\begin{equation}
\epsilon=\left\Vert \vx_{\otimes,\mathcal{I}}- \vx_{l}\right\Vert _{2}
\label{error2}
\end{equation}
where $\vx_{\otimes,\mathcal{I}}$ are only the coefficients $\vx_{\otimes}$ that have multi-indices in $\mathcal{I}$, and where $\vx_{l}$ are the coefficients obtained via least squares on the matrix $\mA_{\ddagger}$. In Figures~\ref{problem_1b}(b-d) we plot $\epsilon$ values for varying total orders with $n/l$ ratios of 1.15, 1.25 and 1.5. Condition numbers associated with the matrices $\mA_{\boxtimes}$ and $\mA_{\ddagger}$ in Figure~\ref{problem_1b} are shown in Figure~\ref{problem_1b_cond}. 
\begin{figure}
\begin{subfigmatrix}{2}
\subfigure[]{\includegraphics{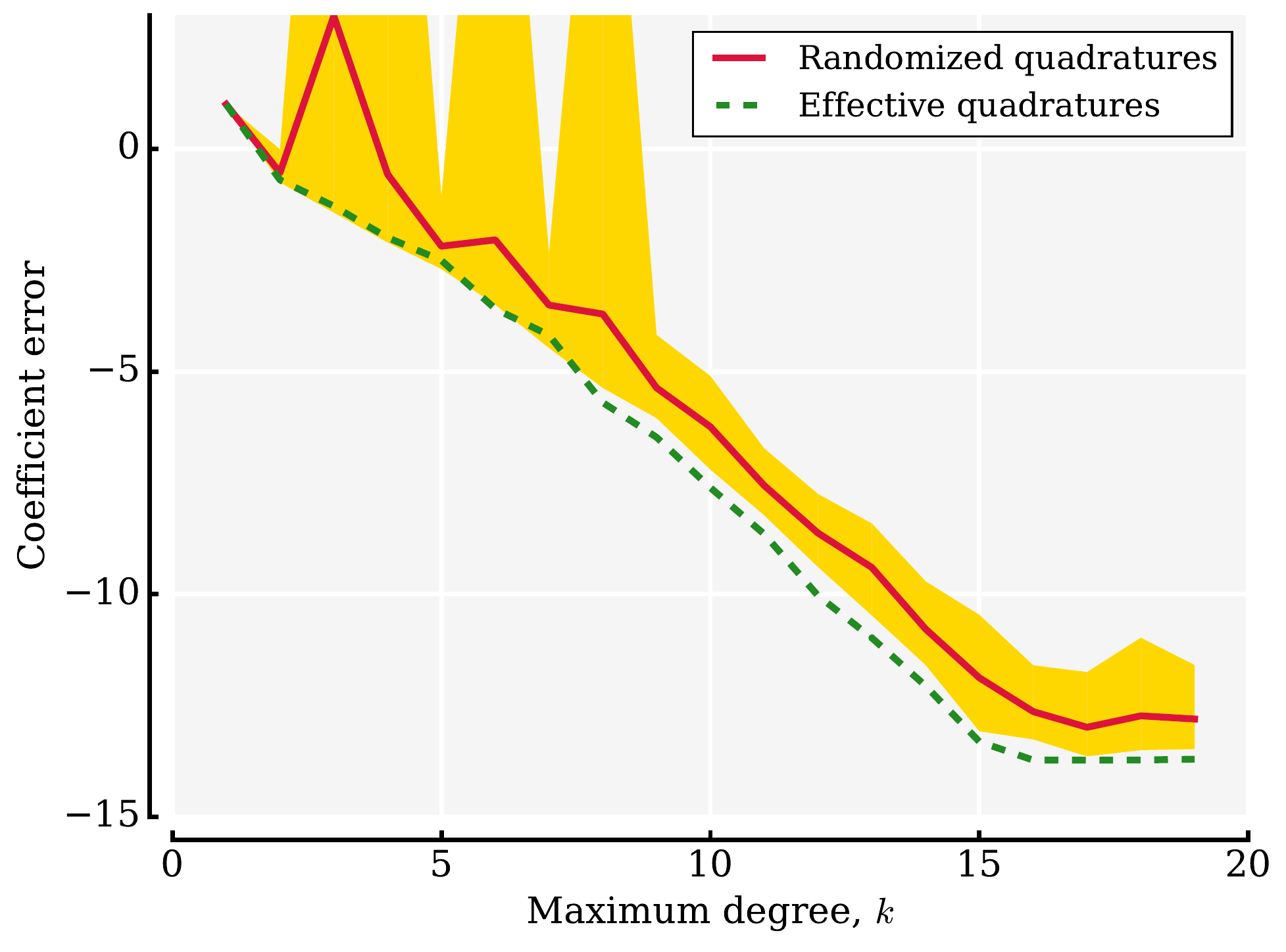}}
\subfigure[]{\includegraphics{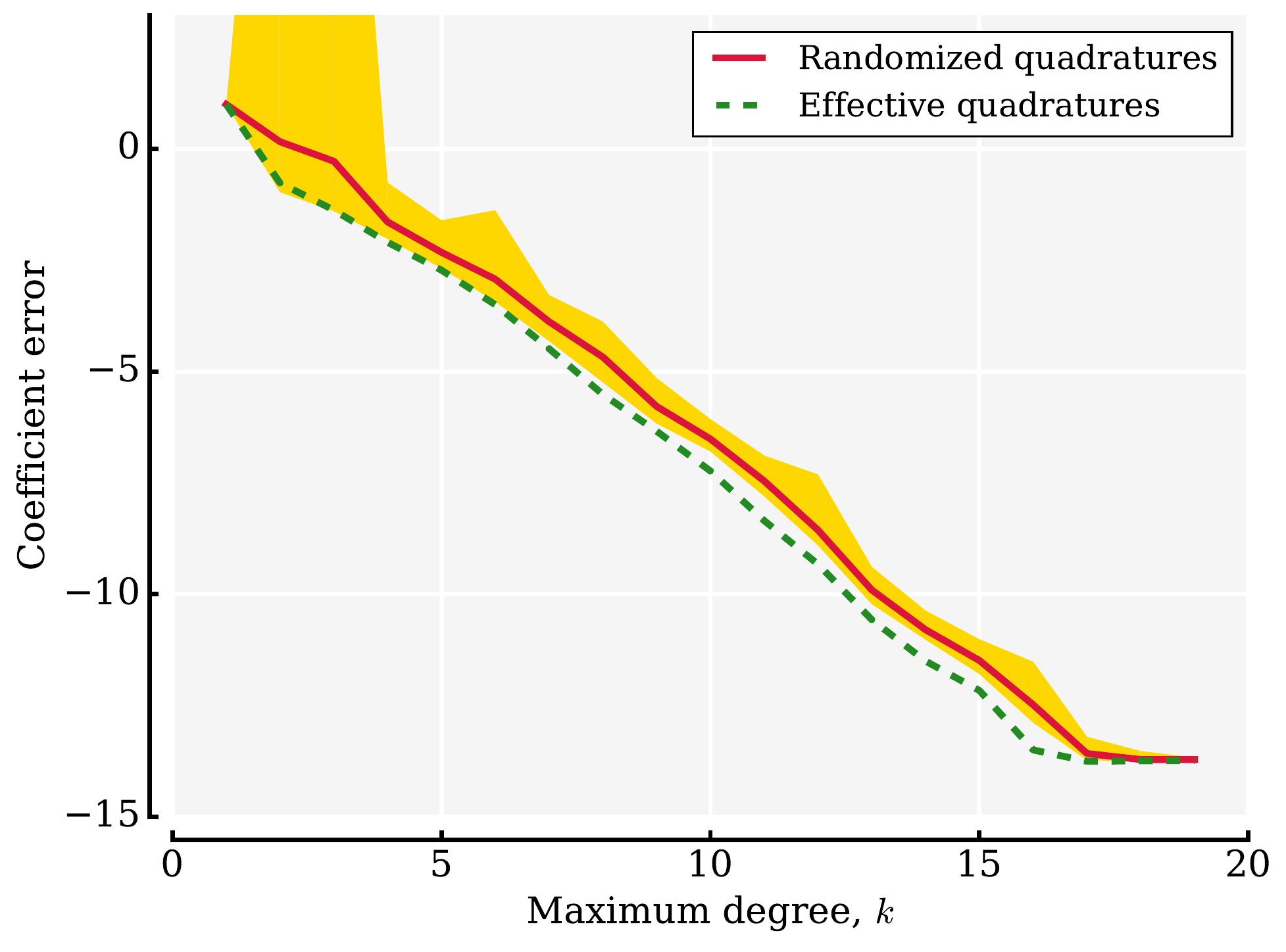}}
\subfigure[]{\includegraphics{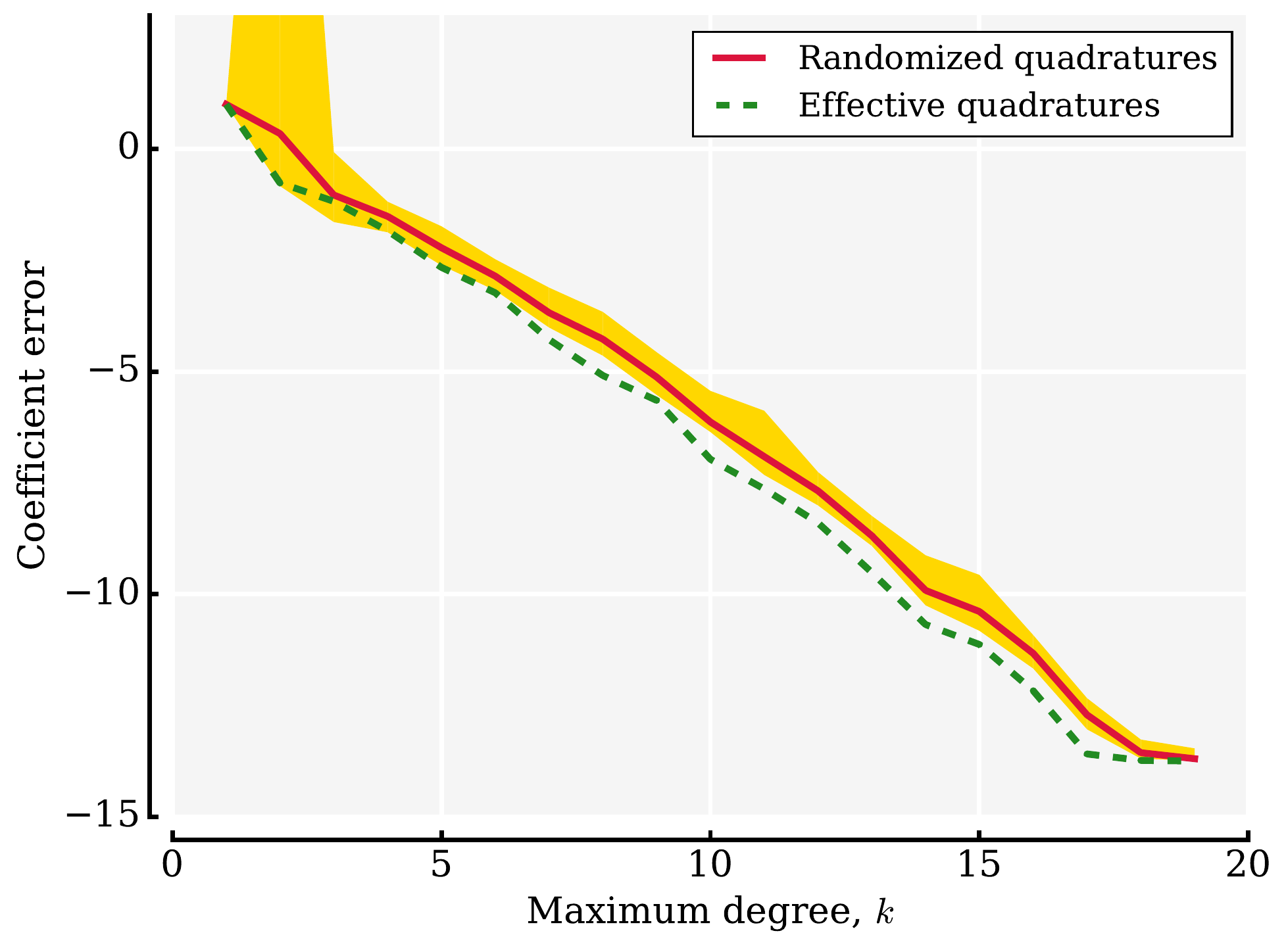}}
\subfigure[]{\includegraphics{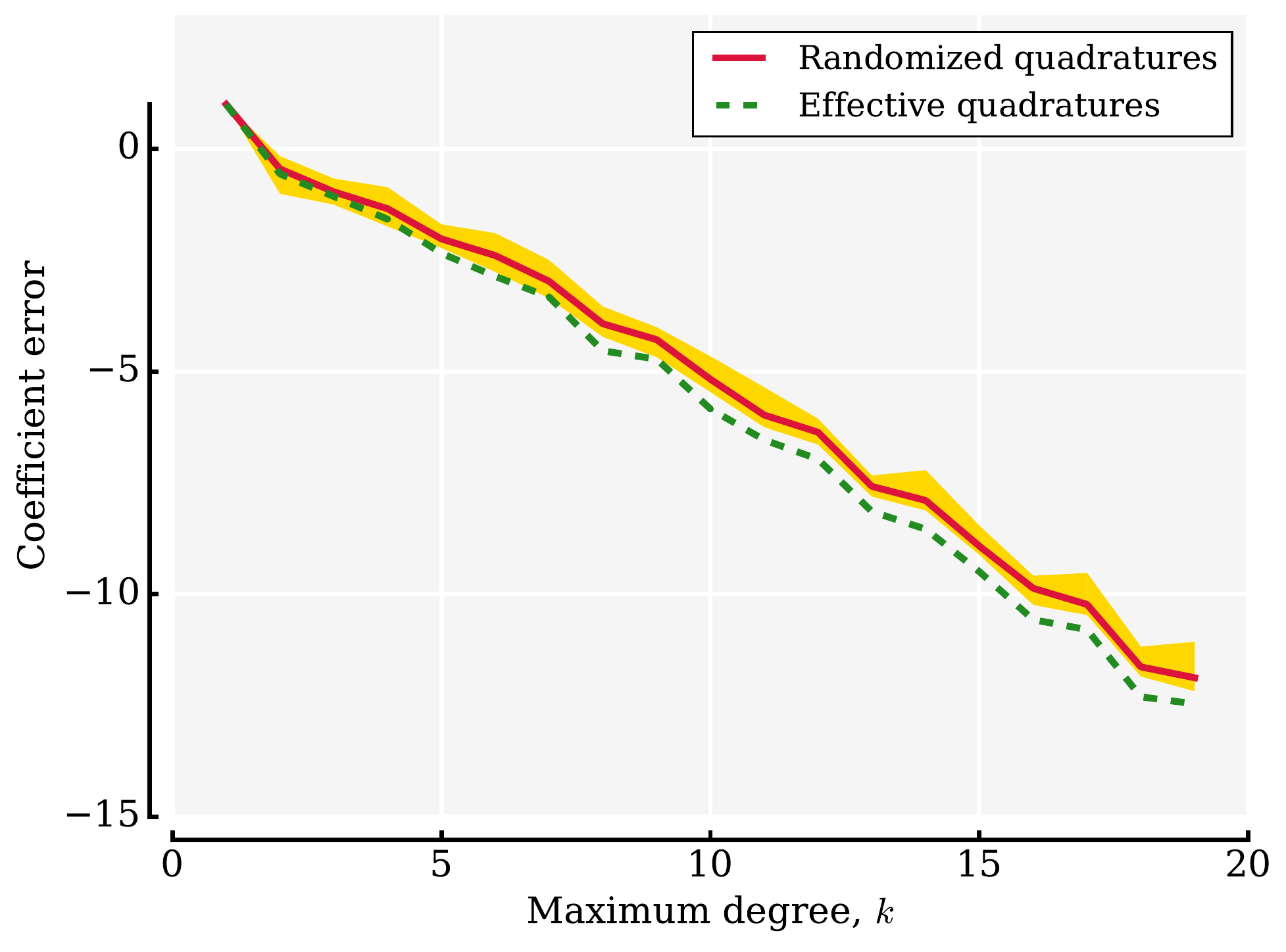}}
\end{subfigmatrix}
\caption{Approximation errors in polynomial coefficients---computed using~\eqref{error1} and~\eqref{error2}---plotted on a base-10 logarithmic scale for the bivariate function, $f(\bm{\zeta}) = exp(\zeta^{(1)} + \zeta^{(2)})$ defined on the $[-1,1]$ hypercube, using Legendre orthonormal polynomials. Polynomial approximations are constructed via least squares using randomized quadratures--- repeated 20 times---and effectively subsampled quadratures. Results are plotted for $n/l$ ratios of (a) 1.0; (b) 1.15; (c) 1.25; (d) 1.5.}
\label{problem_1b}
\end{figure}
\begin{figure}
\begin{subfigmatrix}{2}
\subfigure[]{\includegraphics{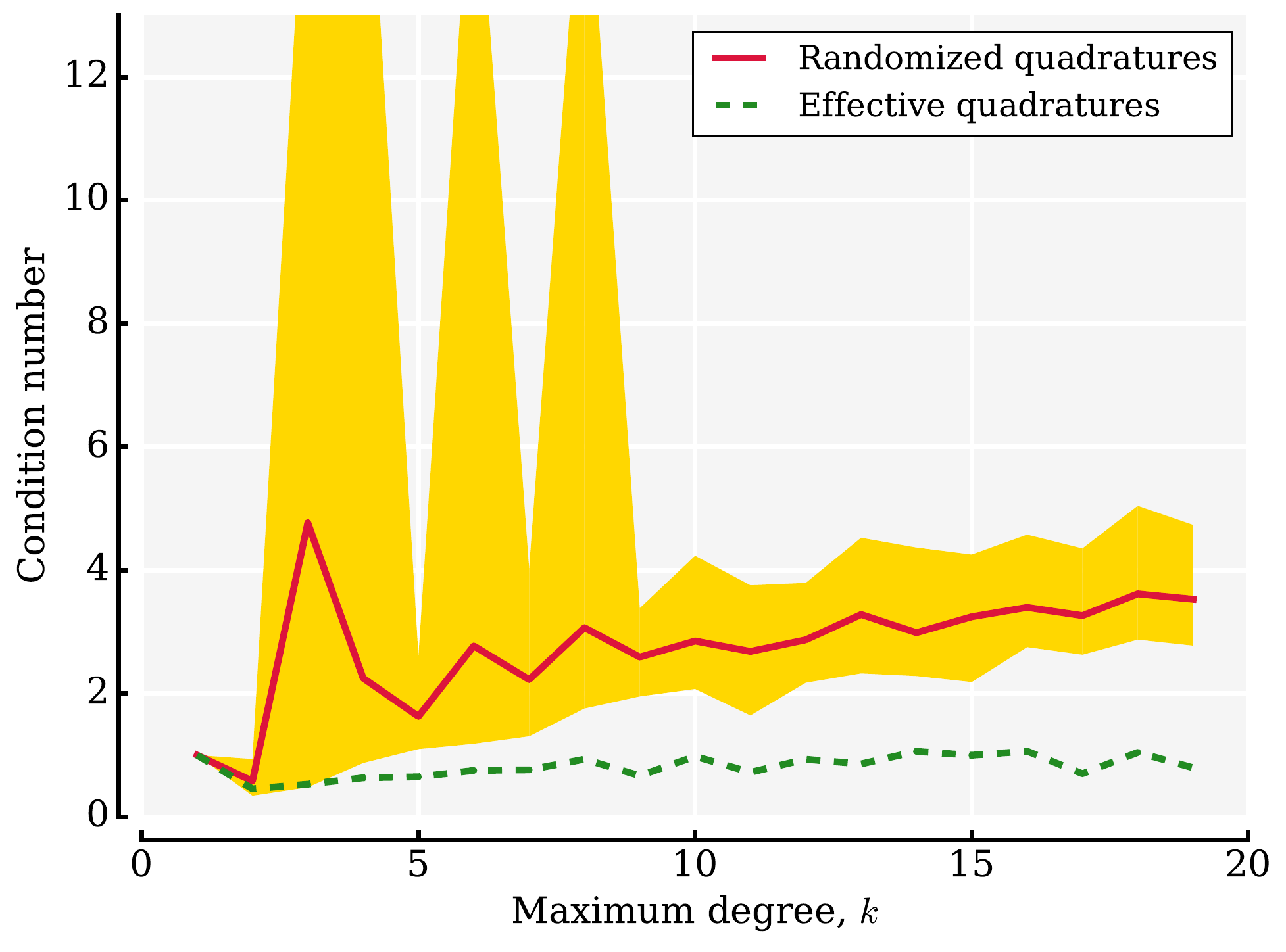}}
\subfigure[]{\includegraphics{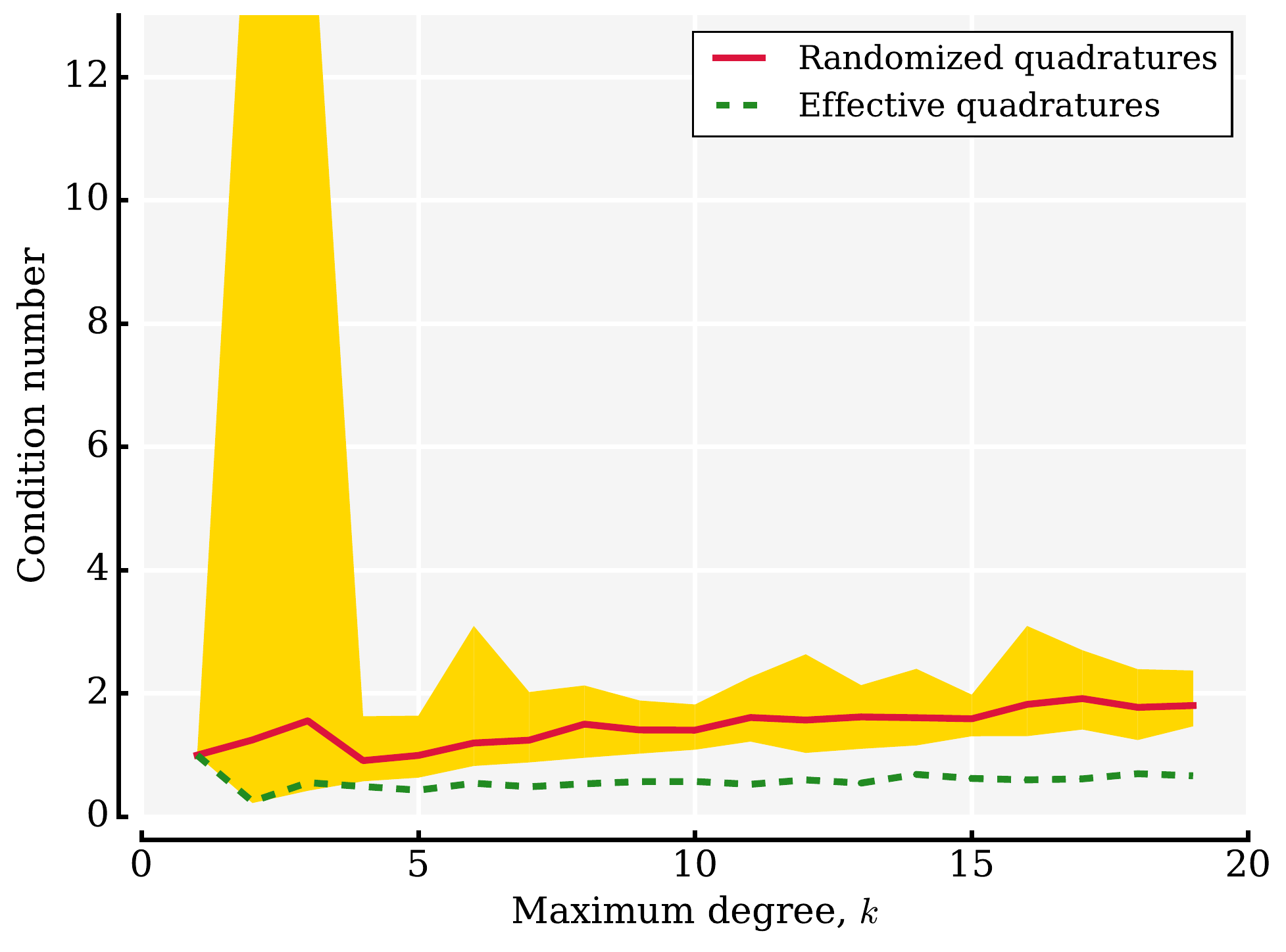}}
\subfigure[]{\includegraphics{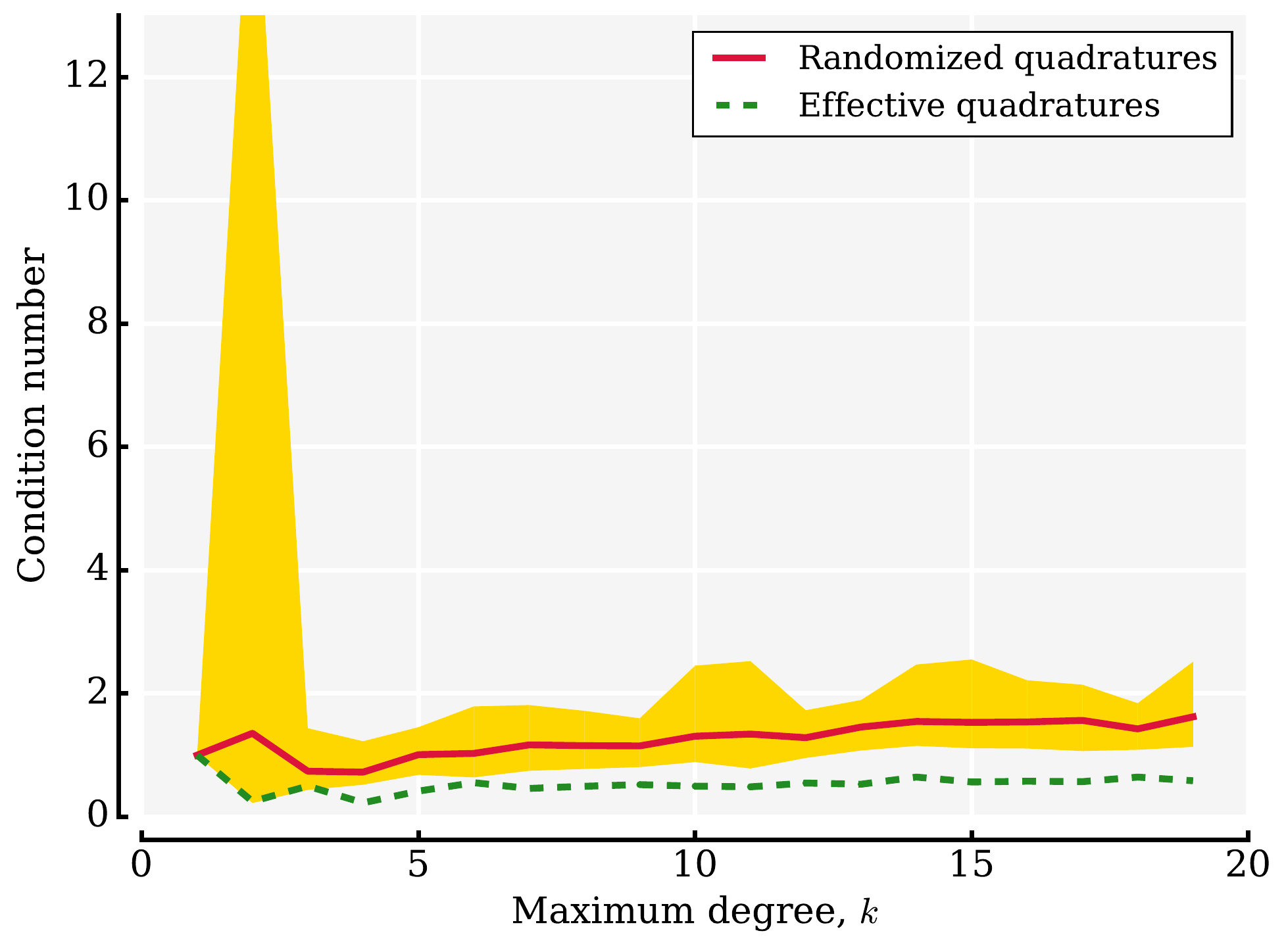}}
\subfigure[]{\includegraphics{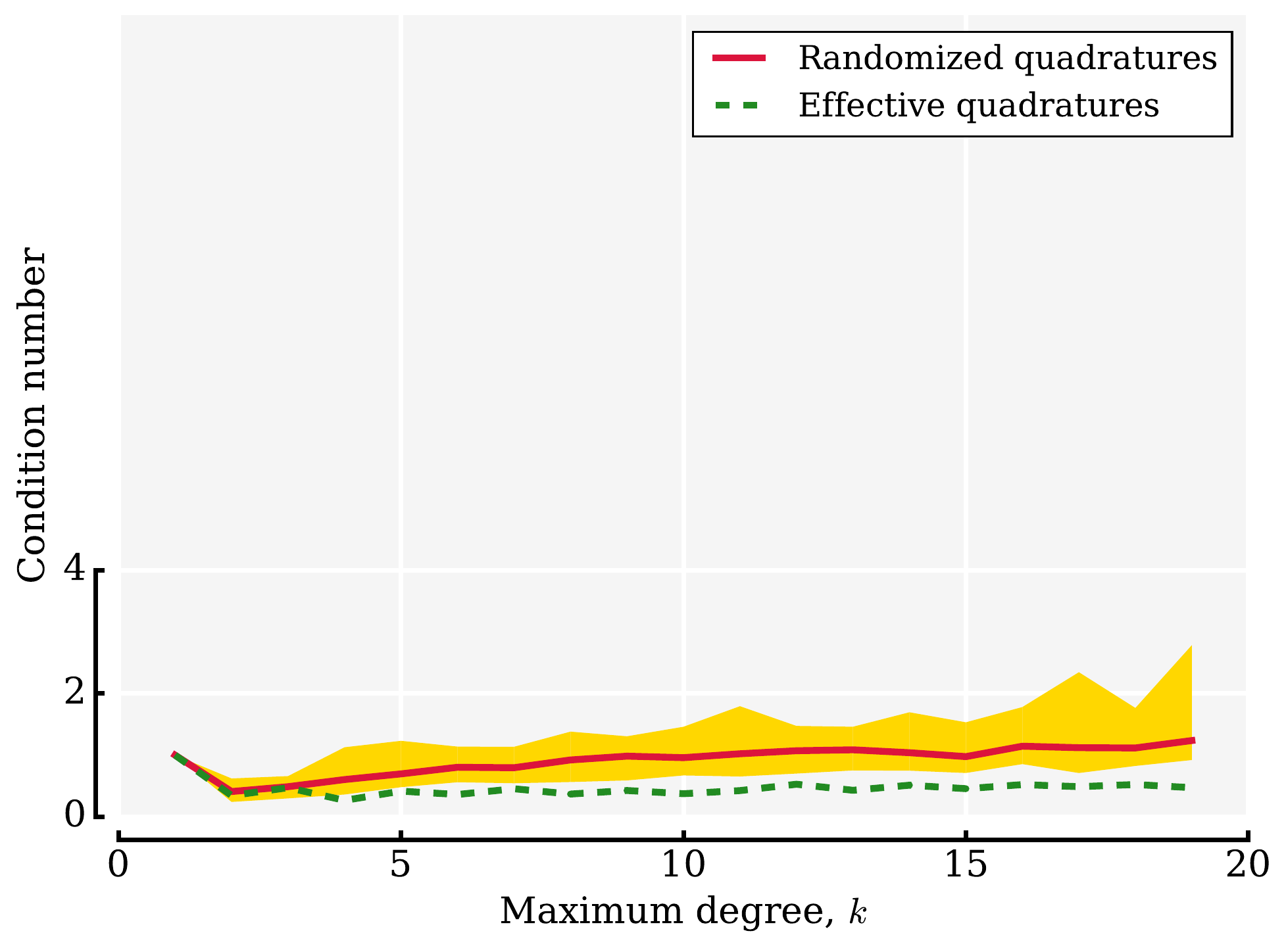}}
\end{subfigmatrix}
\caption{Condition numbers of $\mA_{\boxtimes}$ and $\mA_{\ddagger}$, plotted on a base-10 logarithmic scale for the bivariate function, $f(\bm{\zeta}) = exp(\zeta^{(1)} + \zeta^{(2)})$ defined on the $[-1,1]$ hypercube, using Legendre orthonormal polynomials. Polynomial approximations are constructed via least squares using randomized quadratures---repeated 20 times---and effectively subsampled quadratures. Results are plotted for $n/l$ ratios of (a) 1.0; (b) 1.15; (c) 1.25; (d) 1.5.}
\label{problem_1b_cond}
\end{figure}

First let us consider the results when $n/l$=1.0 as in Figure~\ref{problem_1b}(a) and Figure~\ref{problem_1b_cond}(a). Random draws of rows for low maximum degrees---from $k=2$ to $k=8$---leads to, on average, matrices that are nearly singular. This results in coefficient errors that extremely high as illustrated in Figure~\ref{problem_1b}(a). By pruning down the number of columns, this effect can be reduced as shown in Figures~\ref{problem_1b}(b-d). This also results in better coefficient error estimates, however this comes at the cost of requiring more model evaluations for a given number of coefficients to be estimated. In contrast, our effectively subsampled procedure offers reduced error estimates even at a $n/l$ ratio of 1.0; condition numbers for all $k$ never exceed 10.

\section{Piston model problem}
\label{sec:numex}
In this example, we apply our method on a non-linear model of the cycle time of a piston given in \cite{Kenett}. The piston cycle time $C$ is expressed as
\begin{equation}
C=2\pi\sqrt{\frac{M}{k+S^{2}\frac{P_{0}V_{0}T_{a}}{T_{0}V^{2}}}},
\end{equation}
with 
\begin{equation}
V=\frac{S}{2k}\left(\sqrt{A^{2}+4k\frac{P_{0}V_{0}}{T_{0}}T_{a}}-A\right) \; \; \textrm{and} \; \; A=P_{0}S+19.62M-\frac{kV_{0}}{S}
\end{equation}
which depends on the seven inputs given in Table~\ref{piston_params}. These inputs are uniformly distributed over their respective ranges. Our task here is to contrast \emph{effective quadrature subsampling} with \emph{randomized quadrature subsampling} on a total order basis. 
\begin{table}
 \begin{center}
   \caption{Input parameters and ranges for the piston problem}
	\vspace{0.5 mm}   
   \label{piston_params}
   \begin{tabular}{lll}
    \toprule
    Input parameters & Range & Description \\
    \hline
        \hline
$M$ & $[30,60]$ & Piston weight ($kg$)\\  
$S$ & $[0.005, 0.0020]$ & Piston surface area ($m^2$)\\  
$V_0$ & $[0.002, 0.010]$ & Initial gas volume ($m^3$) \\
$k$ & $[1000, 5000]$ & Spring coefficient ($N/m$) \\
$P_0$ & $[90000, 110000]$ & Atmospheric pressure ($N/m^2$) \\
$T_a$ & $[290, 296]$ & Ambient temperature ($K$) \\
$T_0$ & $[340, 360]$ & Filling gas temperature ($K$) \\ \hline
   \end{tabular}
 \end{center}
\end{table}

To begin, we compute the full tensor grid solution using 5 points in each direction, yielding a total of $5^7=78,125$ function evaluations. We then use these coefficients for computing the errors from both effectively subsampled and randomly subsampled approaches---i.e., $\epsilon$ computed as per~\eqref{error1} and~\eqref{error2}. We also use these coefficients to compute the \emph{full tensor grid} Sobol' indices shown Figure~\ref{Sobol}. 

We run numerical experiments using total order basis with maximum degrees of $k=2,3$ and $4$. Table~\ref{params_points} shows the number of basis terms at each of these k values for three different $n/l$ ratios. Both randomly subsampled and effectively subsampled condition numbers and errors for these numerical experiments are reported in Tables~\ref{results_1}-\ref{results_3}. The randomized experiments were repeated 20 times and their minimum, maximum and mean values are reported.

\begin{table}[H]
 \begin{center}
   \caption{Number of tensor grid points and cardinality of total order basis used for the piston problem numerical experiments}
	\vspace{0.5 mm}   
   \label{params_points}
   \begin{tabular}{l l p{25mm} p{25mm} p{25mm} }
    \toprule
    $k$ & Points & Cardinality \newline $n/l=1.0$ & Cardinality \newline $n/l=1.15$ & Cardinality \newline $n/l=1.25$\\
    \hline
        \hline
2 & $3^7=2,187$ & 36 & 31 & 29 \\  
3 & $4^7=16,384$ & 120 & 104 & 96\\  
4 & $5^7=78,125$ & 330 & 287 & 264\\   \hline
   \end{tabular}
 \end{center}
\end{table}

\begin{table}[H]
 \begin{center}
      \caption{Comparison of error $\epsilon$ in the coefficients and the condition number $\kappa$ of $\mA_{\boxtimes}$ for randomized and effective quadratures for $n/l$=1.00. Reported randomized results are the outcome of 20 repetitions. Here $k$ indicates the maximum degree of the 7D total order polynomial.}
 \vspace{0.5mm}   
    \label{results_1}
  \begin{tabular}{l | cccccc | cc}
    \toprule
    \multirow{2}{*}{$k$}   &
      \multicolumn{6}{c}{Randomized}  \vline & 
      \multicolumn{2}{c}{Effective} \\
      & {min$(\epsilon)$} & {max$(\epsilon)$} & {$\mu(\epsilon)$} & {min$(\kappa)$} & {max$(\kappa)$} & {$\mu(\kappa)$} & {$\epsilon$} & {$\kappa$}\\
      \hline
      \midrule
    2 & 0.051 & 3.061 & 0.448 & 48.244 & 1559.1 & 301.56 & 0.0252 & 5.185  \\
    3 & 0.066 & 3.083 & 0.671 & 115.35 & 12422.1& 2118.2 & 0.01463 & 15.098 \\
    4 & 0.066 & 0.901 & 0.199 & 595.4 &  32782.6 &  2989.0 & 0.0597 & 791.98 \\
    \bottomrule
  \end{tabular} 
      \end{center}
\end{table}

\begin{table}[H]
 \begin{center}
      \caption{Comparison of error $\epsilon$ in the coefficients and the condition number $\kappa$ of $\mA_{\ddagger}$ for randomized and effective quadratures for $n/l$=1.15. Reported randomized results are the outcome of 20 repetitions. Here $k$ indicates the maximum degree of the 7D total order polynomial.}
 \vspace{0.5mm}   
     \label{results_2}
  \begin{tabular}{l | cccccc | cc}
    \toprule
    \multirow{2}{*}{$k$}   &
      \multicolumn{6}{c}{Randomized}  \vline & 
      \multicolumn{2}{c}{Effective} \\
      & {min$(\epsilon)$} & {max$(\epsilon)$} & {$\mu(\epsilon)$} & {min$(\kappa)$} & {max$(\kappa)$} & {$\mu(\kappa)$} & {$\epsilon$} & {$\kappa$}\\
      \hline
      \midrule
    2 & 0.0367 & 0.252 & 0.0779 & 15.02 & 60.83 & 23.41 & 0.0375 & 3.912 \\
    3 & 0.0284 & 0.051& 0.0403 & 22.92 & 36.98 & 29.60 & 0.0196 & 6.821\\
    4 & 0.0151 & 0.023 & 0.0194 & 30.275 & 40.97 & 35.317 & 0.0159 & 33.652\\
    \bottomrule
  \end{tabular} 
      \end{center}
\end{table}

\begin{table}[H]
 \begin{center}
      \caption{Comparison of error $\epsilon$ in the coefficients and the condition number $\kappa$ of $\mA_{\ddagger}$ for randomized and effective quadratures for $n/l$=1.25. Reported randomized results are the outcome of 20 repetitions. Here $k$ indicates the maximum degree of the 7D total order polynomial.}
 \vspace{0.5mm}   
     \label{results_3}
  \begin{tabular}{l | cccccc | cc}
    \toprule
    \multirow{2}{*}{$k$}   &
      \multicolumn{6}{c}{Randomized}  \vline & 
      \multicolumn{2}{c}{Effective} \\
      & {min$(\epsilon)$} & {max$(\epsilon)$} & {$\mu(\epsilon)$} & {min$(\kappa)$} & {max$(\kappa)$} & {$\mu(\kappa)$} & {$\epsilon$} & {$\kappa$}\\
      \hline
      \midrule
    2 & 0.0440 & 0.129  & 0.0719 & 9.809 & 27.043 & 16.008 & 0.0362 & 4.9126 \\
    3 & 0.0244 & 0.0469 & 0.0329 & 14.98 & 23.61 & 19.63 & 0.0186 & 5.7832\\
    4 & 0.0145 & 0.0217 & 0.0168 & 18.72 &25.164 & 21.95 & 0.0163 & 20.507\\
    \bottomrule
  \end{tabular} 
      \end{center}
\end{table}

For all $k$ across all $n/l$ ratios, the effectively subsampled approach yields lower coefficient errors and lower condition numbers compared to the averaged randomized quadrature results. In particular, when $n/l=1.0$ the $\epsilon$ values from effectively subsampled quadratures are 1-3 orders of magnitude below those of the corresponding randomized values. Randomized subsampling in general yields lower coefficient errors at $n/l$ ratios greater than one. For a fixed $k$, as the $n/l$ ratio is increased, both condition numbers and errors are found to drop for both approaches.

For completeness, we compare the Sobol' indices of both the random and effectively subsampled quadratures approaches in Figure~\ref{Sobol}. These indices are computed using the coefficient estimates obtained from the randomized, effectively subsampled and full tensor grid quadrature rule---using the approach in Sudret~\cite{Sudret}, which is detailed in the Appendix of this paper. 
\begin{figure}
\begin{subfigmatrix}{2}
\subfigure[]{\includegraphics{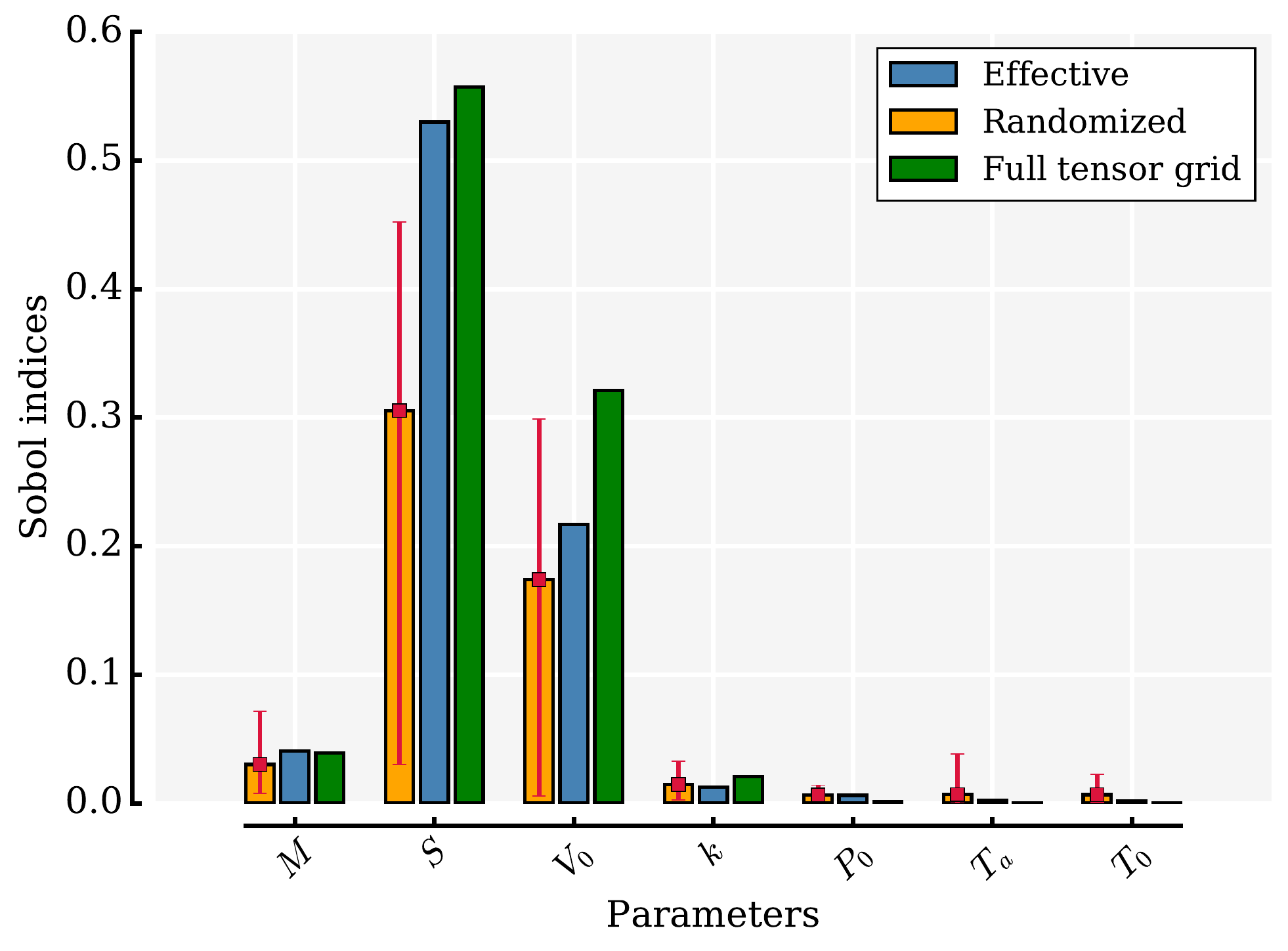}}
\subfigure[]{\includegraphics{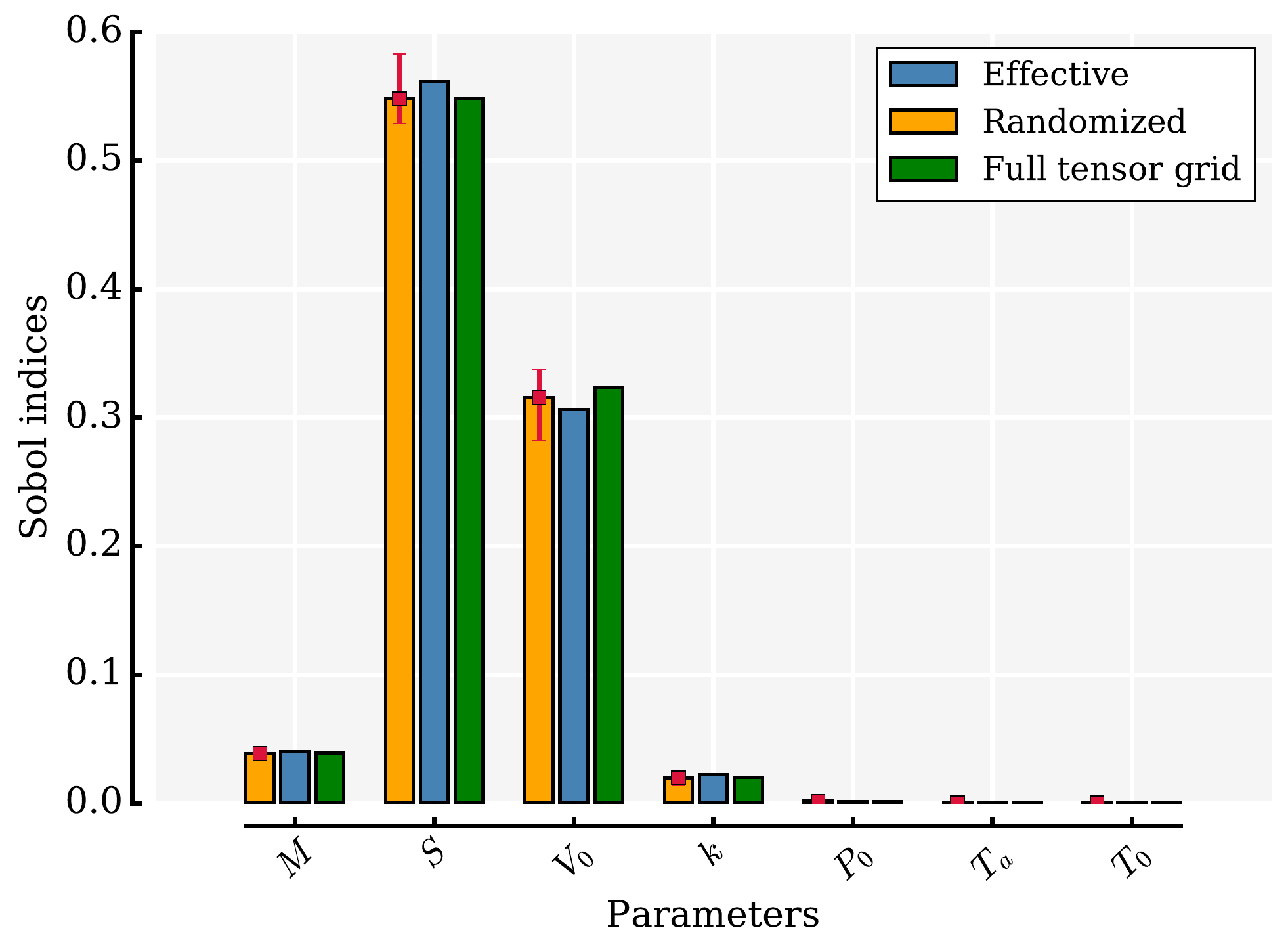}}
\end{subfigmatrix}
\caption{First order Sobol' indices for the seven parameters with $k=4$ for (a) Total order basis with $n/l=1.00$;  (b) Total order basis with $n/l=1.25$. The error bars indicate the minimum and maximum Sobol' indices from the mean (red dot) obtained from 20 repetitions of the randomized approach.}
\label{Sobol}
\end{figure}
These results highlight the point that even with an $n/l=1.0$, effectively subsampled quadratures does yield results that are comparable to those obtained from the full tensor grid.

\section{A problem where effectively subsampled quadratures fails}
\label{sec:fail}
In this section we present an example that illustrates a limitation of our method. Consider the function 
\begin{equation}
f(\bm{\zeta}) = \frac{1}{\left(1 + 50(\zeta_{1} - 0.9)^2 + 50(\zeta_{2} + 0.9)^2     \right)}
\label{prob_equ}
\end{equation}
defined over $\mathcal{R}=[-1,1]^2$. The contours of this function are shown in Figure~\ref{fig7}(a) and its bi-variate Legendre polynomial approximation---obtained by evaluating the function at a tensor grid with 10 points in each direction and computing the coefficients with the corresponding tensor product integration rule---is shown in Figure~\ref{fig7}(b). The function is relatively flat throughout most of its domain, but it exhibits a steep variation in the lower right hand corner around (-0.9, -0.9). 
\begin{figure}
\begin{subfigmatrix}{2}
\subfigure[]{\includegraphics{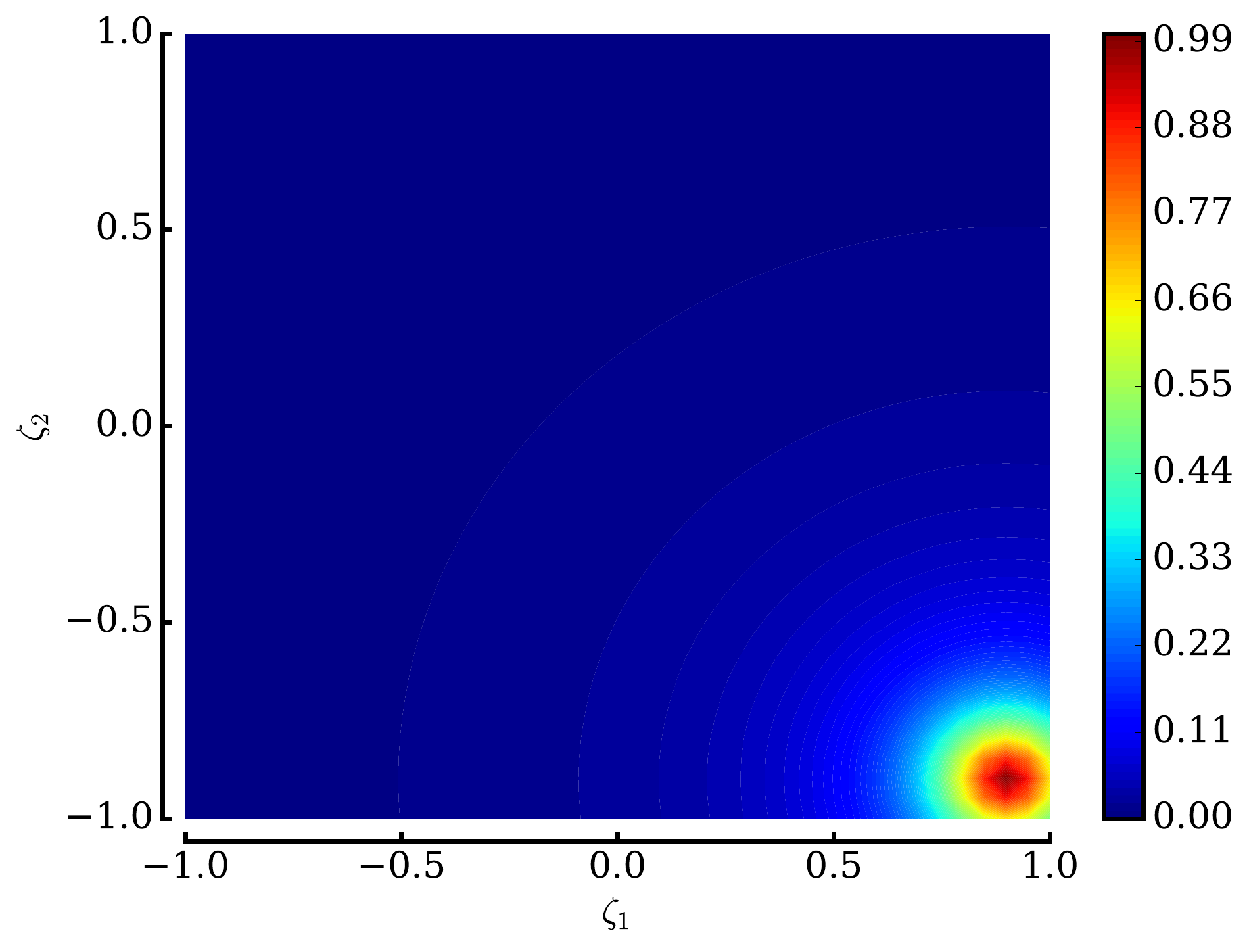}}
\subfigure[]{\includegraphics{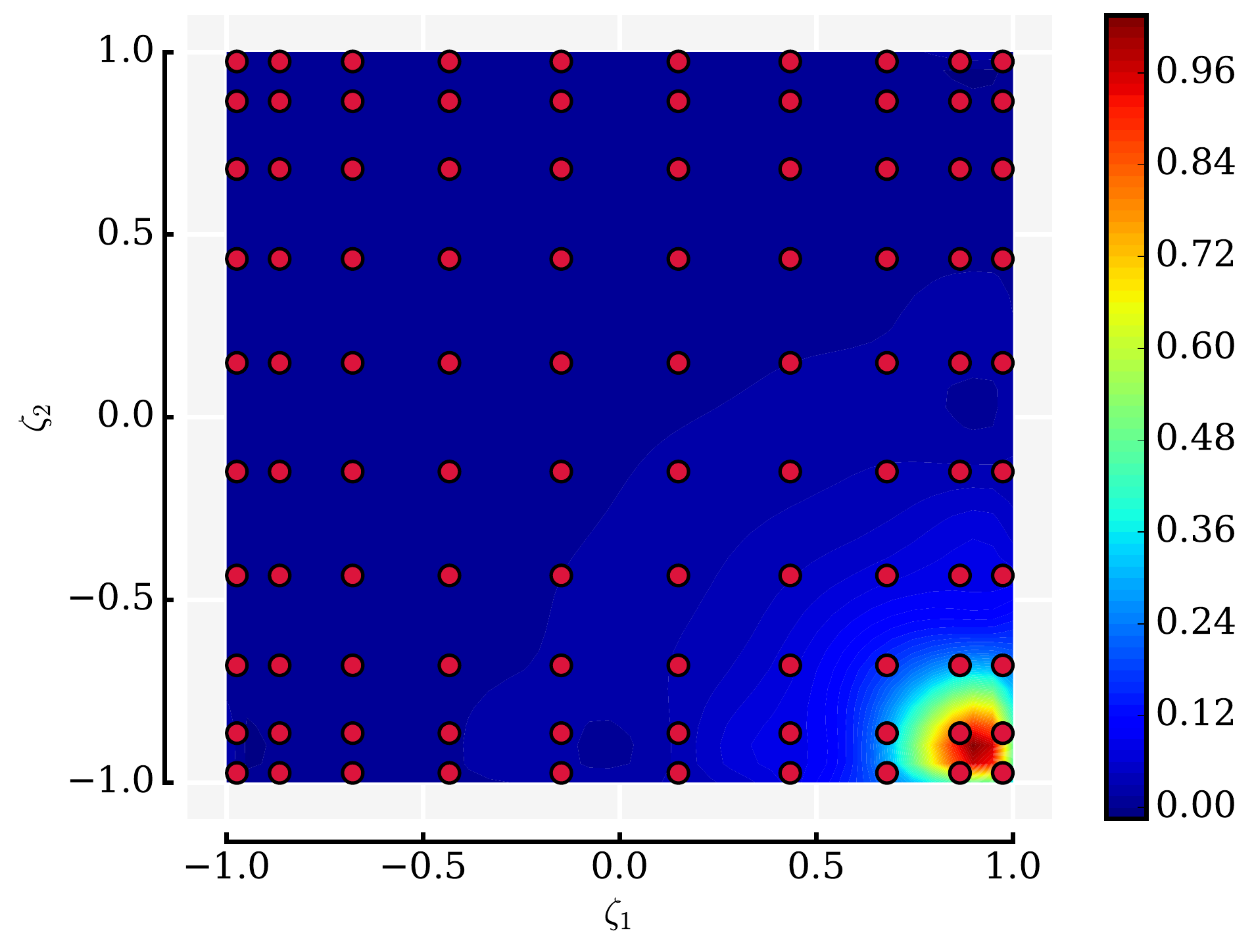}}
\end{subfigmatrix}
\caption{Contour plots for~\eqref{prob_equ} with the real function in (a) and its tensor grid polynomial approximant in (b).}
\label{fig7}
\end{figure}

In Figure~\ref{fig8} we plot the effectively subsampled quadrature approximation with different hyperbolic basis index sets that all have a maximum order of 9 in each direction. Figure~\ref{fig8}(a-b) shows the results for a $q$ factor of 0.3; (c-d) for a $q$ factor of 0.5 and (e-f) for a $q$ factor of 1.0 which is equivalent to a total order basis (For the definition of $q$ see~\eqref{eq:hyperbolic-space}). The tensor grid points in Figures (a,c,e) are given by red circular markers, while the effectively subsampled points---used for generating the polynomial approximation contours---are shown as green ``x'' markers. 

\begin{figure}
\begin{subfigmatrix}{2}
\subfigure[]{\includegraphics{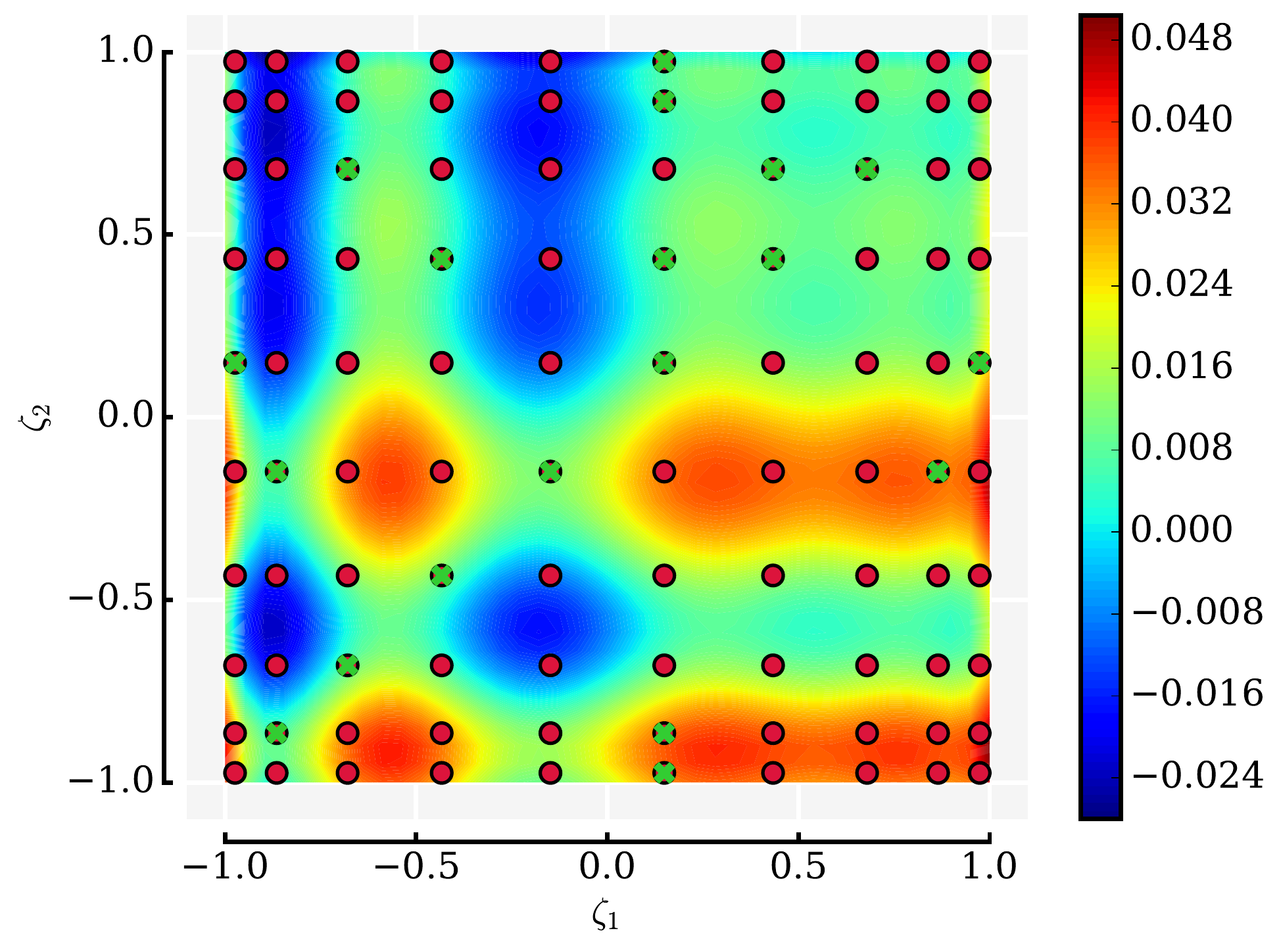}}
\subfigure[]{\includegraphics{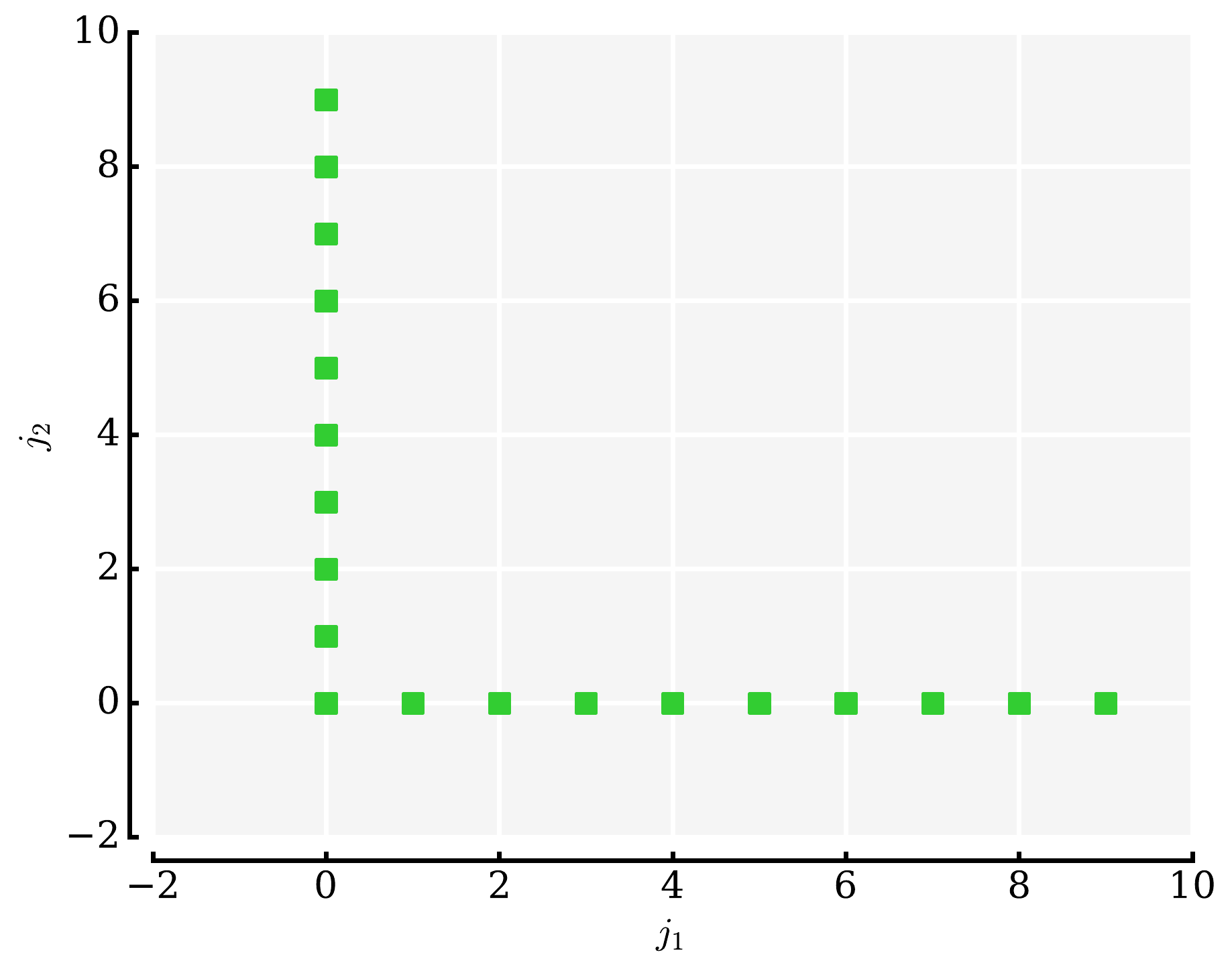}}
\subfigure[]{\includegraphics{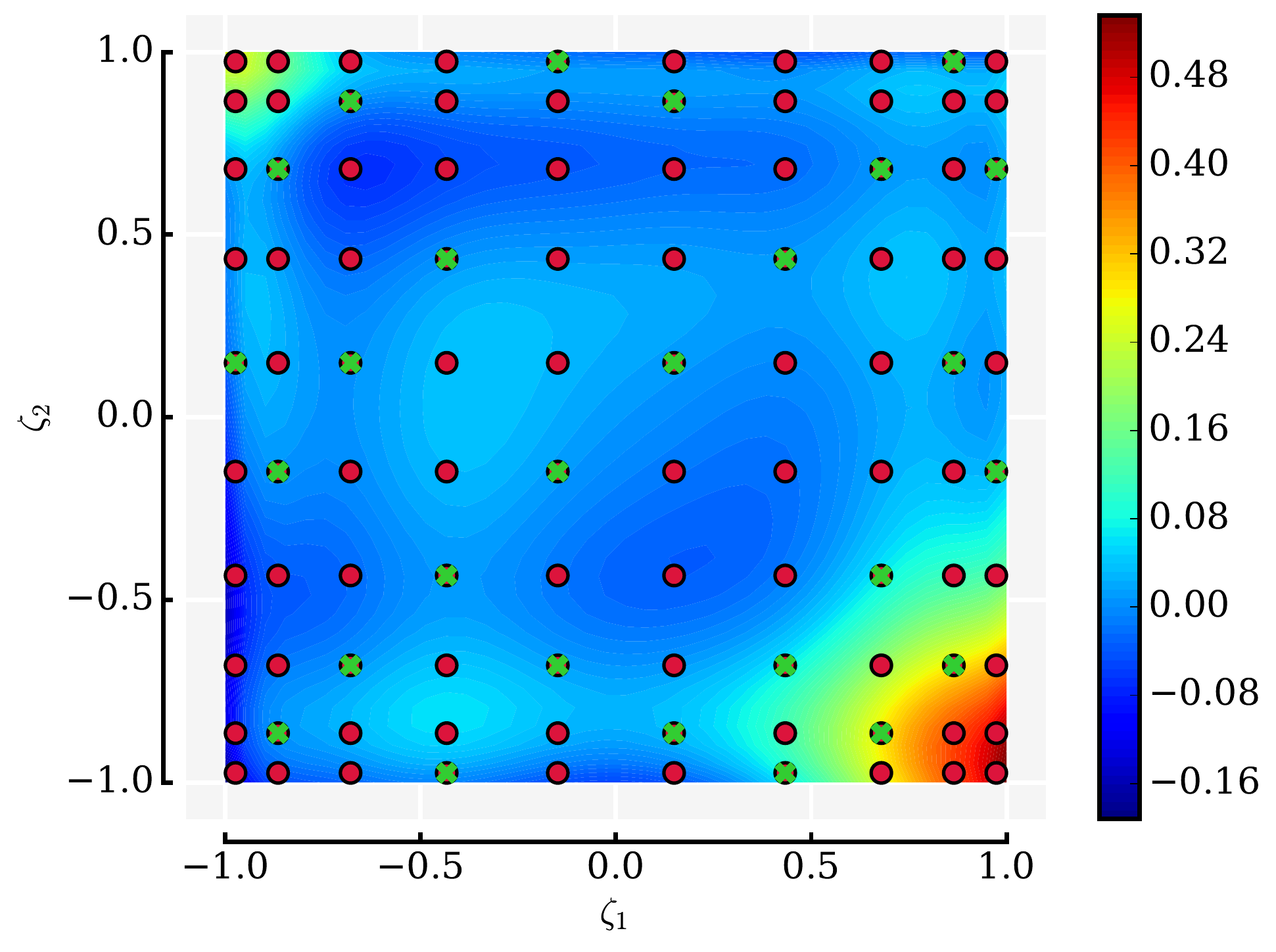}}
\subfigure[]{\includegraphics{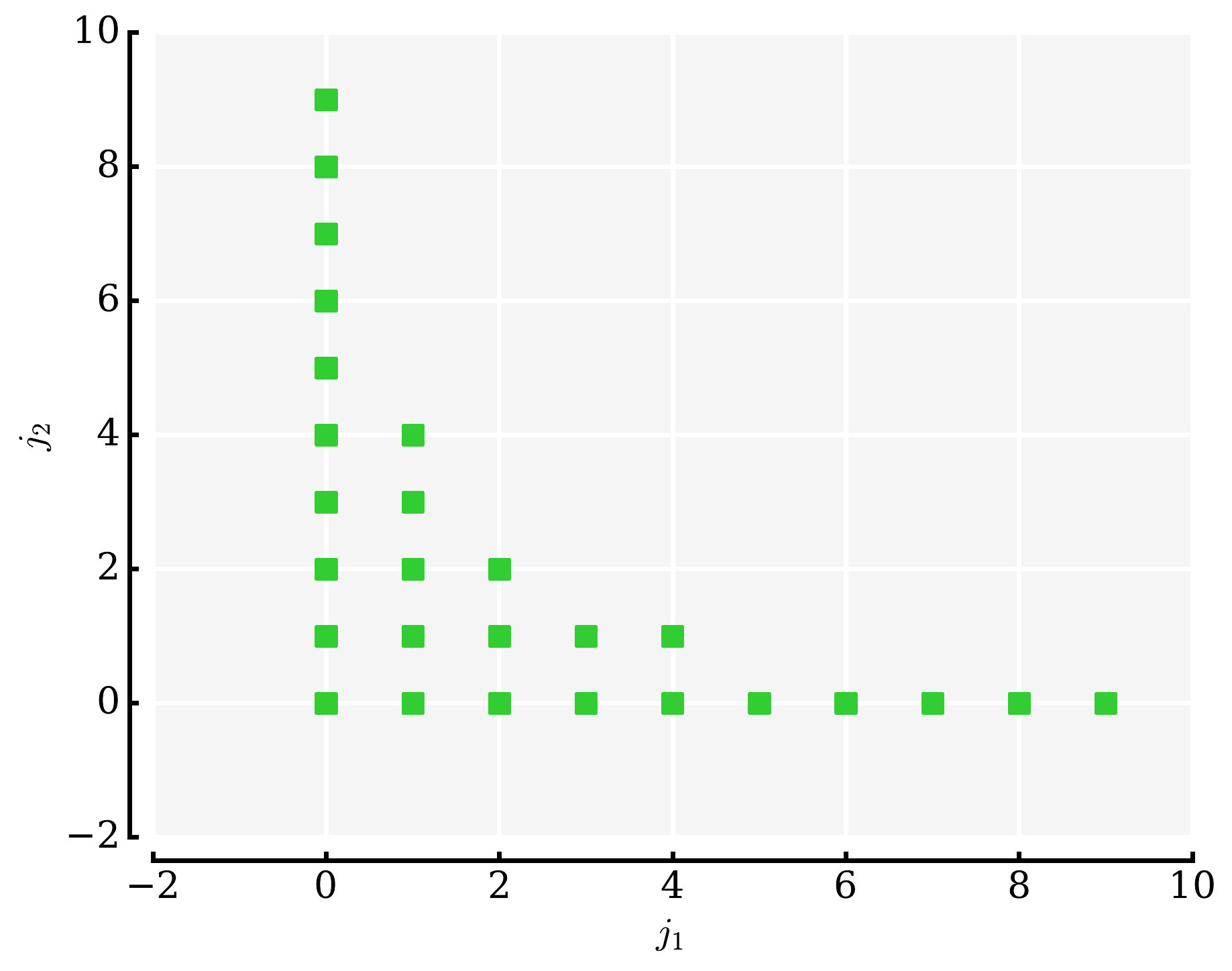}}
\subfigure[]{\includegraphics{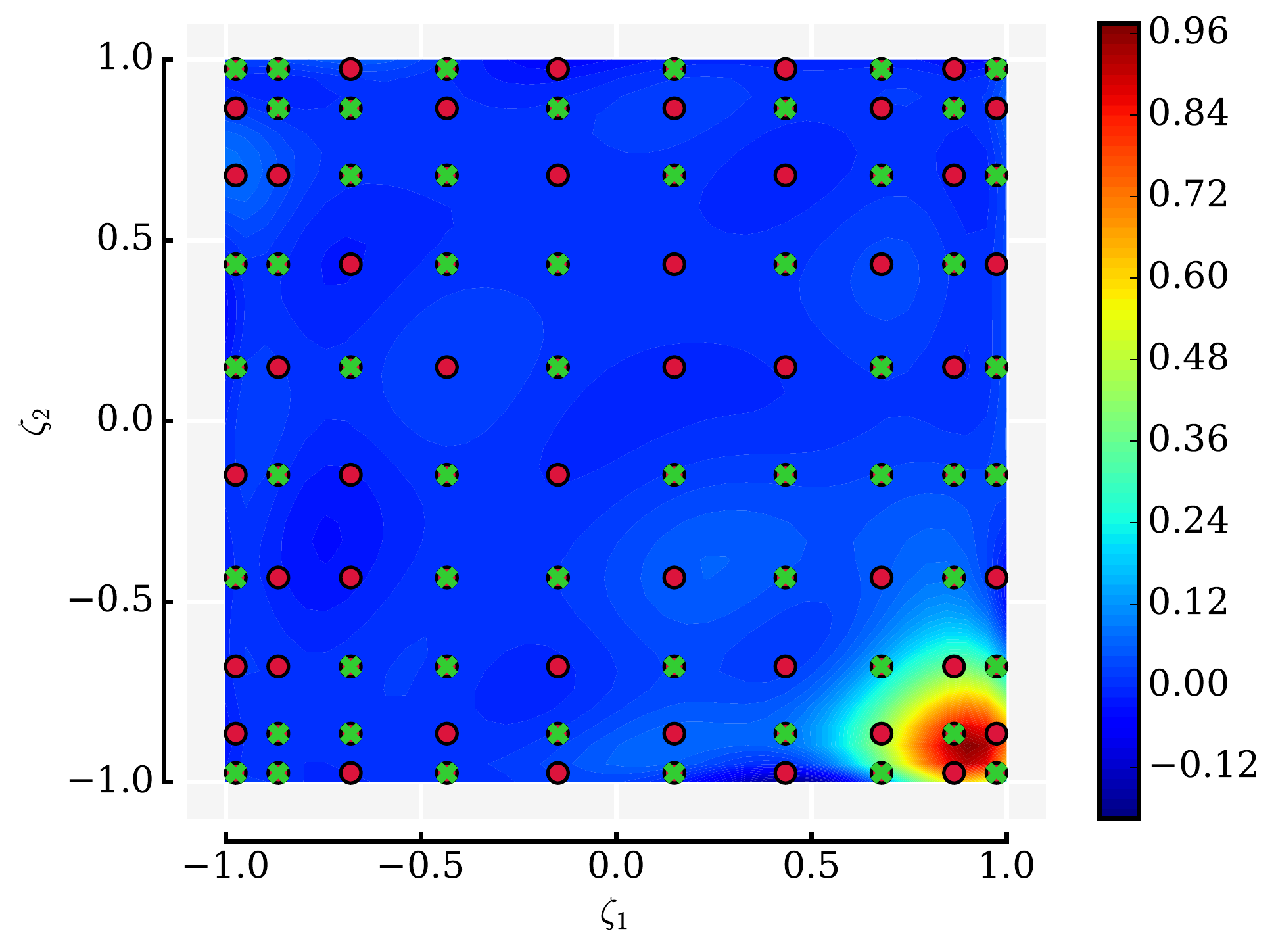}}
\subfigure[]{\includegraphics{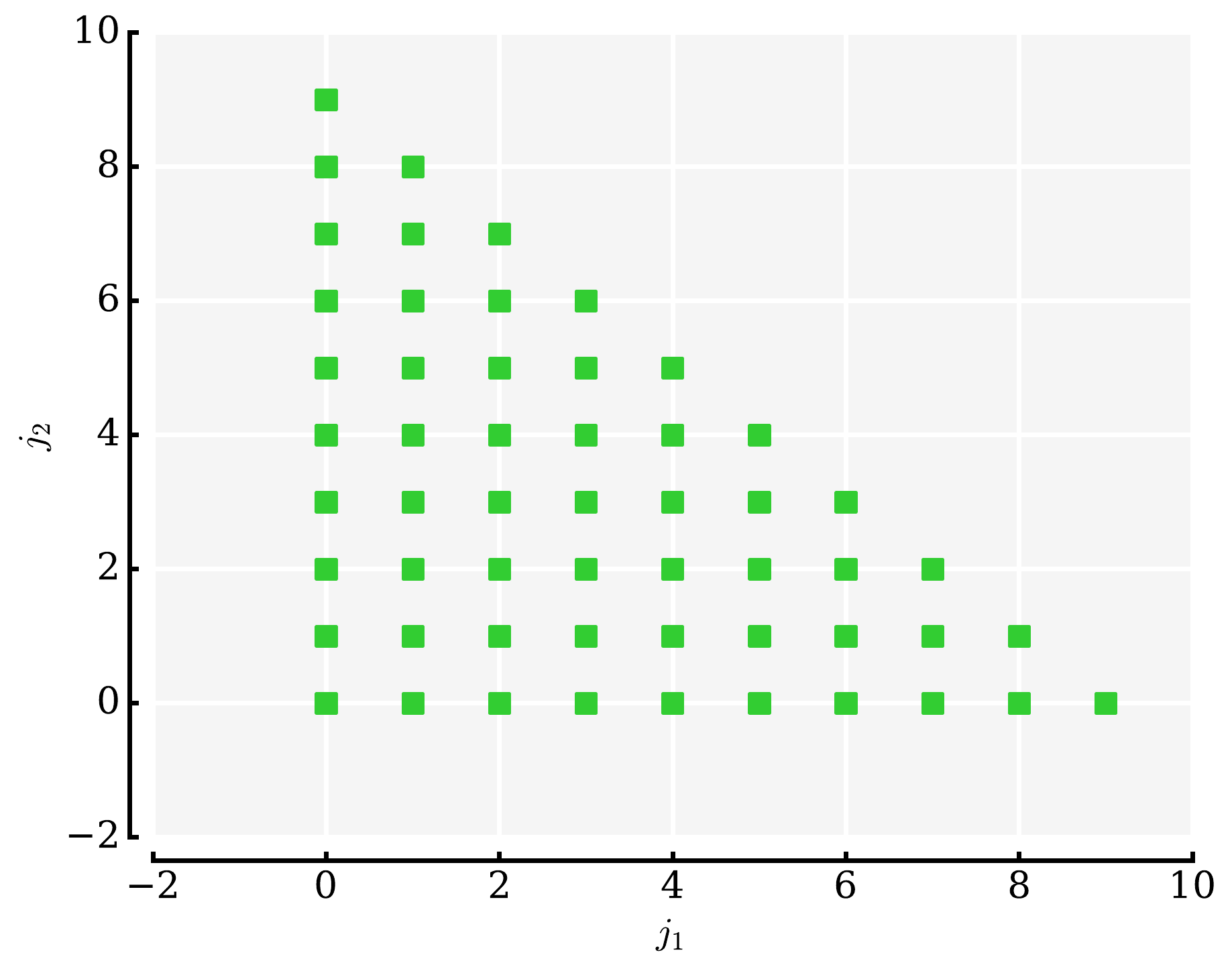}}
\end{subfigmatrix}
\caption{Effectively subsampled quadratures approximations of~\eqref{prob_equ} with the chosen hyperbolic basis with: (a-b) $q=0.3$; (c-d) $q=0.5$; (e-f) $q=1.0$.}
\label{fig8}
\end{figure}

It is clear that in Figure~\ref{fig8}(a) and to a certain extent in (c), the subsampled points do not lie in the vicinity of (-0.9, 0.9), hence the polynomial approximation does a rather poor job of approximating the overall response of the function. This approximation is however improved in Figure~\ref{fig8}(e-f) when using a total order basis, as the subsampled points do lie around (-0.9, 0.9). This illustrates a limitation of effectively subsampled quadratures---i.e., they are based solely on $\rho(\bm{\zeta})$, the choice of the tensor grid and basis---and are invariant to the function response. 

\section{Conclusion}
In this paper, we proposed a new sampling strategy for generating polynomial least squares approximations, titled \emph{effectively subsampled quadratures}. This technique uses a QR column pivoting heuristic for subsampling points from a tensor grid. Numerical results showed the advantages of this technique compared with \emph{randomized quadrature} subsampling. In future work we will investigate the incorporation of gradients. 

\section{Acknowledgements}
The first author would like to thank Paul Constantine for valuable discussions on polynomial least squares, Alireza Doostan for insightful exchanges on compressed sensing strategies with pseudospectral approximations and Tiziano Ghisu for his feedback on the randomized results. Thanks are also due to Gregorio Quintana-Ort\'{i} and Mario Arioli for their insights into QR with column pivoting. The authors also thank the reviewers for their suggestions and comments, which improved the overall quality of this manuscript. 

\section*{Appendix: Computing Sobol' indices}
Consider the polynomial approximation of $f(s)$ provided earlier (see~\eqref{main_poly_equ})
\begin{equation}
f\left(s\right)\approx g(s) = \sum_{\bm{j} \in\mathcal{J}}^{n}x_{\bm{j}} \bm{\psi _j} \left(s\right).
\end{equation}
By using the orthogonality properties of the polynomial basis, the mean $\mu$ and variance $\sigma^2$ of $g(s)$ can be expresed solely using the expansion's coefficients (see \cite{smith2013uncertainty} page 210)
\begin{equation}
\mu=x_{\bm{0}}, \; \; \; \; \sigma^{2}=\sum_{\bm{j}\in\mathcal{J},  \; \bm{j}\neq\bm{0}}x_{\bm{j}}^{2}.
\end{equation}
Recall that Sobol' indices represent a fraction of the total variance that is attributed to each input variable (the first order Sobol' indices) or combinations thereof (higher order Sobol' indices). Let $\mathcal{J}_{\bm{s}}$ be the set of multi-indices that depend only on the subset of variables $\bm{s}=\left\{ j_{1},\ldots,j_{s}\right\}$, i.e.,
\begin{equation}
\mathcal{J}_{\bm{s}}=\left\{ \bm{j} \in\mathbb{N}^{d} \; : \; l\in \bm{s}\Leftrightarrow j_{l}\neq0\right\}.
\end{equation}
The first order partial variances $\sigma_i^2$ are then obtained by summing up the square of the coefficients in $\mathcal{J}_{\bm{s}}$
\begin{equation}
\sigma_{i}^{2}=\sum_{\bm{j} \in\mathcal{J}_{i}} x_{\bm{j}}^{2},   \; \; \; \; \mathcal{J}_{i}=\left\{ \bm{j}\in\mathbb{N}^{d}:j_{i}>0\right\} ,
\end{equation}
and the higher order variances $\sigma_{\left\{ j_{1},\ldots,j_{s}\right\}}^2$ can be written as
\begin{equation}
\sigma_{\bm{s}}^2= \sum_{\bm{j} \in\mathcal{J}_{\bm{s}}  } x_{\bm{j}}^{2},   \; \; \; \; \; \mathcal{J}_{\left\{ j_{1},\ldots,j_{s}\right\}}=\left\{ \bm{j}\in\mathbb{N}^{d}: l \in \bm{s} \Leftrightarrow j_{l}>0\right\}.
\end{equation}
The first and higher order Sobol indices are then given by
\begin{equation}
S_{i}=\frac{\sigma_{i}^{2}}{\sigma^{2}}\; \;  \textrm{and} \; \; S_{\bm{s}}=\frac{\sigma_{\bm{s}}^{2}}{\sigma^{2}}
\end{equation}
respectively (for further details see~\cite{Sudret}).

\end{document}